\documentclass[11pt]{article}
\usepackage{eurosym}
\usepackage{amssymb}
\usepackage{amsmath}
\usepackage{amsfonts}
\usepackage{boxedminipage}
\usepackage{color}
\usepackage{fancyhdr}
\usepackage{makeidx}
\usepackage{graphicx}
\usepackage{mathrsfs}
\usepackage{cases}
\usepackage{geometry}

\providecommand{\U}[1]{\protect\rule{.1in}{.1in}}
\newtheorem{proposition}{Proposition}[section]
\newtheorem{theorem}[proposition]{Theorem}

\newtheorem{lemma}[proposition]{Lemma}

\newtheorem{definition}[proposition]{Definition}
\newtheorem{remark}[proposition]{Remark}

\newtheorem{condition}[proposition]{Condition}

\numberwithin{equation}{section}
\numberwithin{proposition}{section}
\newenvironment{proof}[1][Proof]{\noindent\textbf{#1.} }{\ \rule{0.5em}{0.5em}}

\makeindex
 \parskip        \baselineskip
\begin{document}

\title{Large Deviations for Stochastic Partial Differential Equations Driven
by a Poisson Random Measure}
\author{Amarjit Budhiraja\thanks{%
Research supported in part by the National Science Foundation (DMS-1004418,
DMS-1016441), the Army Research Office (W911NF-10-1-0158) and the US-Israel
Binational Science Foundation (Grant 2008466).}, Jiang Chen and Paul Dupuis%
\thanks{%
Research supported in part by the National Science Foundation (DMS-1008331),
and the Air Force Office of Scientific Research (FA9550-09-1-0378,
FA9550-12-1-0399), and the Army Research Office (W911NF-09-1-0155).}}
\maketitle

\begin{abstract}
Stochastic partial differential equations driven by Poisson random measures
(PRM) have been proposed as models for many different physical systems,
where they are viewed as a refinement of a corresponding noiseless partial
differential equation (PDE). A systematic framework for the study of
probabilities of deviations of the stochastic PDE from the deterministic PDE
is through the theory of large deviations. The goal of this work is to
develop the large deviation theory for small Poisson noise perturbations of
a general class of deterministic infinite dimensional models. Although the
analogous questions for finite dimensional systems have been well studied,
there are currently no general results in the infinite dimensional setting.
This is in part due to the fact that in this setting solutions may have
little spatial regularity, and thus classical approximation methods for
large deviation analysis become intractable. The approach taken here, which
is based on a variational representation for nonnegative functionals of
general PRM, reduces the proof of the large deviation principle to
establishing basic qualitative properties for controlled analogues of the
underlying stochastic system. As an illustration of the general theory, we
consider a particular system that models the spread of a pollutant in a
waterway.\medskip

\medskip

\end{abstract}




\section{Introduction}

Stochastic partial differential equations driven by Poisson random measures
arise in many different fields. For example, they have been used to develop
models for neuronal activity that account for synaptic impulses occurring
randomly, both in time and at different locations of a spatially extended
neuron. Other applications arise in chemical reaction-diffusion systems and
stochastic turbulence models. The starting point in all these application
areas are deterministic partial differential equations (PDE) that capture
the underlying physics. One then develops a stochastic evolution model
driven by a suitable Poisson noise process to take into account random
inputs or effects to the nominal deterministic dynamics. In typical settings
the solutions of these stochastic evolution equations are not smooth. In
fact in many applications of interest they are not even random fields (that
is, function valued), and therefore an appropriate framework is given
through the theory of generalized functions. A systematic theory of
existence and uniqueness of solutions (both weak and pathwise) for such
stochastic partial differential equations (SPDE) driven by Poisson random
measures has been developed in \cite{Kallianpur95}. Our objective in this
work is to study some large deviation problems associated with such
stochastic systems.

Large deviation properties of SPDE driven by infinite dimensional Brownian
motions (e.g. Brownian sheets) have been extensively studied. In a typical
such setting one considers a small parameter multiplying the noise term and
is interested in asymptotic probabilities of non-nominal behavior as the
parameter approaches zero. This is the classical Freidlin-Wentzell problem
that has been studied in numerous papers (see the references in \cite%
{buddupmar}). Earlier works on this family of problems were based on ideas
of \cite{Azen} and relied on discretizations and other approximations
combined with `super-exponential closeness' probability estimates. For many
models of interest, particularly those arising from fluid dynamics and
turbulence, developing the required exponential probability estimates is a
daunting task and consequently simpler alternative methods are of interest.
In recent years an approach based on certain variational representation
formulas for moments of nonnegative functionals of Brownian motions \cite%
{buddupmar} has been increasingly used for the study of the small noise
large deviation problem for Brownian motion driven infinite dimensional
systems \cite{BeMi, buddupmar, BuDuMa2, ChMi, DuDuGa, DuMi, Liu,MaSrSu,
ReZh1, ReZh2, RoZhZh, SrSu, WaDu, YaHo, zha3, zha4}. The main appealing
feature of this approach is that it completely bypasses
approximation/discretization arguments and exponential probability
estimates, and in their place essentially requires a basic qualitative
understanding of existence, uniqueness and stability (under `bounded'
perturbations) of certain controlled analogues of the underlying stochastic
dynamical system of interest.

Large deviation results for finite dimensional stochastic differential
equations with a Poisson noise term has been studied by several authors \cite%
{Wentz,LiPu,Dupuis97,deAc}. For infinite dimensional models with jumps, very
little is available. One exception is the paper \cite{RocZha} that obtains
large deviation results for an Ornstein-Uhlenbeck type process driven by an
infinite dimensional L\'{e}vy noise. One reason there is relatively little
work in the Poisson noise setting is that approximation arguments that one
uses for Brownian noise models become much more onerous in the Poisson
setting, and for general infinite dimensional models the approach of \cite%
{Azen} becomes intractable.

With the expectation that it would prove useful for the study of large
deviations for SPDEs driven by Poisson Random Measures (PRM), the paper \cite%
{BDM09} developed a variational representation, for moments of non negative
functionals of PRMs, which is analogous to the representation given in \cite%
{buddup, buddupmar} for the Brownian motion case. The paper \cite{BDM09}
also obtained large deviation results for a basic model of a finite
dimensional jump-diffusion to illustrate the applicability of this
variational representation for the study of large deviation problems for
models with jumps. However the feasibility of this approach for the study of
complex infinite dimensional stochastic dynamical systems driven by Poisson
random measures has not been addressed to date.

The goal of this work is to demonstrate that the approach based on
variational representations that has been very successful for obtaining
large deviation results for system driven by Brownian noises works equally
well for SPDE models driven by PRMs. As in the Brownian case we study the
small noise problem, which in the Poisson setting means that the jump
intensity is $O(\epsilon ^{-1})$ and jump sizes are $O(\epsilon )$, where $%
\epsilon $ is a small parameter. We consider a rather general family of
models of the form
\begin{equation}
X_{t}^{\epsilon }=X_{0}^{\epsilon }+\int_{0}^{t}A(s,X_{s}^{\epsilon
})ds+\epsilon \int_{0}^{t}\int_{\mathbb{X}}G(s,X_{s-}^{\epsilon },v)\tilde{N}%
^{\epsilon ^{-1}}(dsdv),  \label{eq:sdeg}
\end{equation}%
where $N^{\epsilon ^{-1}}$ is a Poisson random measure on $[0,T]\times
\mathbb{X}$ with a $\sigma $-finite mean measure $\epsilon ^{-1}\lambda
_{T}\otimes \nu $, $\lambda _{T}$ is the Lebesgue measure on $[0,T]$ and $%
\tilde{N}^{\epsilon ^{-1}}([0,t]\times B)=N^{\epsilon ^{-1}}([0,t]\times
B)-\epsilon ^{-1}t\nu (B)$, $\forall B\in \mathcal{B}(\mathbb{X})$ with $\nu
(B)<\infty $, is the compensated Poisson random measure.

As noted previously, a key issue with a Poisson noise model is the selection
of an appropriate state space, since it is natural and often convenient for
there to be little spatial regularity. However, many of these foundational
issues have been satisfactorily resolved in \cite{Kallianpur95}, where
pathwise existence and uniqueness of SPDE of the form (\ref{eq:sdeg}) are
treated under rather general conditions. In the framework of \cite%
{Kallianpur95} solutions lie in the space of RCLL trajectories that take
values in the dual of a suitable nuclear space. This framework covers many
specific application settings that have been studied in the literature
(e.g., spatially extended neuron models, chemical reaction-diffusion
systems, etc.). In a parallel with the case of Brownian noise, one finds
that the estimates needed for establishing the well-posedness of the
equation are precisely the ones that are key for the proof of the large
deviation result as well.

The paper is organized as follows. We begin in Section \ref{prelim} with
some background results. The variational representation from \cite{BDM09} is
recalled and also a general large deviation result established in that paper
is presented. Also summarized are basic existence and uniqueness results
from \cite{Kallianpur95} for SPDEs with solutions in the duals of Countably
Hilbertian Nuclear Spaces (CHNS). In Section \ref{sec:three} we study the
small noise problem and state verifiable conditions on the model data in (%
\ref{eq:sdeg}) under which a large deviations principle holds. Section \ref%
{sec:eg} considers a particular system designed to model the spread of a
pollutant in a waterway, and verifies all the conditions assumed on (\ref%
{eq:sdeg}). Finally, the Appendix collects some auxiliary results.

The following notation will be used. For a topological space $\mathcal{E}$,
denote the corresponding Borel $\sigma $-field by $\mathcal{B}(\mathcal{E})$%
. We will use the symbol \textquotedblleft $\Rightarrow $\textquotedblright\
to denote convergence in distribution. Let $\mathbb{N},\mathbb{N}_{0},%
\mathbb{Z},\mathbb{R},\mathbb{R}_{+},\mathbb{R}^{d}$ denote the set of
positive integers, non-negative integers, integers, real numbers, positive
real numbers, and $d$-dimensional real vectors respectively. For a Polish
space $\mathbb{X}$, denote by $C([0,T]:\mathbb{X})$ and $D([0,T]:\mathbb{X})$
the space of continuous functions and right continuous functions with left
limits from $[0,T]$ to $\mathbb{X}$, endowed with the uniform and Skorokhod
topology, respectively. For a metric space $\mathcal{E}$, denote by $M_{b}(%
\mathcal{E})$ and $C_{b}(\mathcal{E})$ the space of real bounded $\mathcal{B}%
(\mathcal{E})/\mathcal{B}(\mathbb{R})$-measurable maps and real continuous
bounded functions respectively. For a measure $\nu $ on $\mathcal{E}$ and a
Hilbert space $H$, let $L^{2}(\mathcal{E},\nu ;H)$ denote the space of
measurable functions $f$ from $\mathcal{E}$ to $H$ such that $\int_{\mathcal{%
E}}||f(v)||^{2}\nu (dv)<\infty $, where $||\cdot ||$ is the norm on $H$. For
a function $x:[0,T]\rightarrow \mathcal{E}$, we use the notation $x_{t}$ and
$x(t)$ interchangeably for the evaluation of $x$ at $t\in \lbrack 0,T]$. A
similar convention will be followed for stochastic processes. We say a
collection $\{X^{\epsilon }\}$ of $\mathcal{E}$-valued random variables is
tight if the distributions of $X^{\epsilon }$ are tight in $\mathcal{P}(%
\mathcal{E})$ (the space of probability measures on $\mathcal{E}$).

%
%

A function $I: \mathcal{E} \rightarrow [0,\infty]$ is called a rate function
on $\mathcal{E}$, if for each $M < \infty$ the level set $\{x \in \mathcal{E}%
:I(x) \le M\}$ is a compact subset of $\mathcal{E}$.
%
%
%
A sequence $\{X^{\epsilon }\}$ of $\mathcal{E}$ valued random variables is
said to satisfy the Laplace principle upper bound (respectively lower bound)
on $\mathcal{E}$ with rate function $I$ if for all $h\in C_{b}(\mathcal{E})$
\begin{equation*}
\limsup_{\epsilon \rightarrow 0}\epsilon \log \mathbb{E}\left\{ \exp \left[ -%
\frac{1}{\epsilon }h(X^{\epsilon })\right] \right\} \leq -\inf_{x\in
\mathcal{E}}\{h(x)+I(x)\},
\end{equation*}%
and, respectively,
\begin{equation*}
\liminf_{\epsilon \rightarrow 0}\epsilon \log \mathbb{E}\left\{ \exp \left[ -%
\frac{1}{\epsilon }h(X^{\epsilon })\right] \right\} \geq -\inf_{x\in
\mathcal{E}}\{h(x)+I(x)\}.
\end{equation*}%
The Laplace principle is said to hold for $\{X^{\epsilon }\}$ with rate
function $I$ if both the Laplace upper and lower bounds hold. It is well
known that when $\mathcal{E}$ is a Polish space, the family $\{X^\epsilon\}$
satisfies the Laplace principle upper (respectively lower) bound with a rate
function $I$ on $\mathcal{E}$ if and only if $\{X^\epsilon\}$ satisfies the
large deviation upper (respectively lower) bound for all closed sets
(respectively open sets) with the rate function $I$. For a proof of this
statement we refer the reader to Section 1.2 of \cite{Dupuis97}.

\section{Preliminaries}

\label{prelim}

\subsection{Poisson Random Measure and a Variational Representation}

\label{Sec:PRM}

Let $\mathbb{X}$ be a locally compact Polish space. Let $\mathcal{M}_{FC}(%
\mathbb{X})$ be the space of all measures $\nu$ on $(\mathbb{X}, \mathcal{B}(%
\mathbb{X}))$ such that $\nu(K) < \infty$ for every compact $K$ in $\mathbb{X%
}$. Endow $\mathcal{M}_{FC}(\mathbb{X})$ with the weakest topology such that
for every $f \in C_c(\mathbb{X})$ (the space of continuous functions with
compact support), the function $\nu \mapsto \langle f, \nu \rangle = \int_{%
\mathbb{X}} f(u) d \nu(u)$, $\nu \in \mathcal{M}_{FC}(\mathbb{X})$ is
continuous. This topology can be metrized such that $\mathcal{M}_{FC}(%
\mathbb{X})$ is a Polish space (see e.g. \cite{BDM09}). Fix $T \in (0,
\infty)$ and let $\mathbb{X}_T = [0,T]\times \mathbb{X}$. Fix a measure $\nu
\in \mathcal{M}_{FC}(\mathbb{X})$, and let $\nu_T = \lambda_T \otimes \nu$,
where $\lambda_T$ is Lebesgue measure on $[0, T]$.

We recall that a Poisson random measure $\mathbf{n}$ on $\mathbb{X}_T$ with
mean measure (or intensity measure) $\nu_T$ is a $\mathcal{M}_{FC}(\mathbb{X}%
_T)$ valued random variable such that for each $B \in \mathcal{B} (\mathbb{X}%
_T)$ with $\nu_T(B) < \infty$, $\mathbf{n}(B)$ is Poisson distributed with
mean $\nu_T(B)$ and for disjoint $B_1,...,B_k \in \mathcal{B} (\mathbb{X}_T)$%
, $\mathbf{n}(B_1),..., \mathbf{n}(B_k)$ are mutually independent random
variables (cf. \cite{Ikeda}). Denote by $\mathbb{P}$ the measure induced by $%
\mathbf{n}$ on $(\mathcal{M}_{FC}(\mathbb{X}_T), \mathcal{B}(\mathcal{M}%
_{FC}(\mathbb{X}_T)))$. Then letting $\mathbb{M} = \mathcal{M}_{FC}(\mathbb{X%
}_T)$, $\mathbb{P}$ is the unique probability measure on $(\mathbb{M},
\mathcal{B}(\mathbb{M}))$ under which the canonical map, $N : \mathbb{M}%
\rightarrow\mathbb{M},N(m) \doteq m$, is a Poisson random measure with
intensity measure $\nu_T$. With applications to large deviations in mind, we
also consider, for $\theta > 0$, probability measures $\mathbb{P}_\theta$ on
$(\mathbb{M}, \mathcal{B}(\mathbb{M}))$ under which $N$ is a Poisson random
measure with intensity $\theta\nu_T$ . The corresponding expectation
operators will be denoted by $\mathbb{E}$ and $\mathbb{E}_\theta$,
respectively. We now present a variational representation, obtained in \cite%
{BDM09}, for $-\log \mathbb{E}_\theta( \exp [-F(N)])$, where $F \in M_b(%
\mathbb{M})$, in terms of a Poisson random measure constructed on a larger
space. We begin by describing this construction.

The analysis of large deviation properties for a process such as (\ref%
{eq:sdeg}) is simplified considerably by a convenient control representation
for the exponential integrals appearing in the Laplace principle. In
contrast with the case of Brownian motion, the formulation of a useful
representation is not immediate for Poisson noise. With a Poisson random
measure, one needs a control that alters the intensity at time $t$ and for
jump type $x$ from that of the underlying PRM to essentially any value in $%
[0,\infty )$ in a non-anticipating fashion. To accommodate this form of
control, we augment the space of jump times and jump types by a variable $%
r\in \lbrack 0,\infty )$, and consider in place of the original PRM one
whose intensity is a product of $\nu _{T}$ and Lebesgue measure on $r$. The
desired jump intensities can then be obtained by \textquotedblleft
thinning\textquotedblright\ this variable.

Thus we let $\mathbb{Y}=\mathbb{X}\times \lbrack 0,\infty )$ and $\mathbb{Y}%
_{T}=[0,T]\times \mathbb{Y}$. Let $\bar{\mathbb{M}}=\mathcal{M}_{FC}(\mathbb{%
Y}_{T})$ and let $\bar{\mathbb{P}}$ be the unique probability measure on $(%
\bar{\mathbb{M}},\mathcal{B}(\bar{\mathbb{M}}))$ under which the canonical
map, $\bar{N}:\bar{\mathbb{M}}\rightarrow \bar{\mathbb{M}},\bar{N}(m)\doteq
m $, is a Poisson random measure with intensity measure $\bar{\nu}%
_{T}=\lambda _{T}\otimes \nu \otimes \lambda _{\infty }$, with $\lambda
_{\infty }$ Lebesgue measure on $[0,\infty )$. The corresponding expectation
operator will be denoted by $\bar{\mathbb{E}}$. Let $\mathcal{F}_{t}\doteq
\sigma \{\bar{N}((0,s]\times A):0\leq s\leq t,A\in \mathcal{B}(\mathbb{Y})\}$%
, and let $\bar{\mathcal{F}}_{t}$ denote the completion under $\bar{\mathbb{P%
}}$. We denote by $\bar{\mathcal{P}}$ the predictable $\sigma $-field on $%
[0,T]\times \bar{\mathbb{M}}$ with the filtration $\{\bar{\mathcal{F}}%
_{t}:0\leq t\leq T\}$ on $(\bar{\mathbb{M}},\mathcal{B}(\bar{\mathbb{M}}))$.
Let $\bar{\mathcal{A}}$ be the class of all $(\bar{\mathcal{P}}\otimes
\mathcal{B}(\mathbb{X}))/\mathcal{B}[0,\infty )$-measurable maps $\varphi :%
\mathbb{X}_{T}\times \bar{\mathbb{M}}\rightarrow \lbrack 0,\infty )$. For $%
\varphi \in \bar{\mathcal{A}}$, define a counting process $N^{\varphi }$ on $%
\mathbb{X}_{T}$ by
\begin{equation}
N^{\varphi }((0,t]\times U)=\int_{(0,t]\times U}\int_{(0,\infty
)}1_{[0,\varphi (s,x)]}(r)\bar{N}(dsdxdr),\quad t\in \lbrack 0,T],U\in
\mathcal{B}(\mathbb{X}).  \label{Eqn: control}
\end{equation}%
$N^{\varphi }$ is then the controlled random measure, with $\varphi $
selecting the intensity for the points at location $x$ and time $s$, in a
possibly random but non-anticipating way. When $\varphi (s,x,\bar{m})\equiv
\theta \in (0,\infty )$, we write $N^{\varphi }=N^{\theta }$. Note that $%
N^{\theta }$ has the same distribution with respect to $\bar{\mathbb{P}}$ as
$N$ has with respect to $\mathbb{P}_{\theta }$. Define $l:[0,\infty
)\rightarrow \lbrack 0,\infty )$ by
\begin{equation*}
l(r)=r\log r-r+1,\quad r\in \lbrack 0,\infty ).
\end{equation*}%
For any $\varphi \in \bar{\mathcal{A}}$ the quantity
\begin{equation}
L_{T}(\varphi )=\int_{\mathbb{X}_{T}}l(\varphi (t,x,\omega ))\nu _{T}(dtdx)
\label{Ltdef}
\end{equation}%
is well defined as a $[0,\infty ]$-valued random variable. The following is
a representation formula proved in \cite{BDM09}.

\begin{theorem}
\label{VR: PRM} Let $F \in M_b(\mathbb{M})$. Then, for $\theta > 0$,
\begin{equation*}
-\log \mathbb{E}_\theta(e^{-F(N)}) = -\log \bar{\mathbb{E}}%
(e^{-F(N^\theta)}) = \inf_{\varphi \in \bar{\mathcal{A}}}\bar{\mathbb{E}}%
\left[\theta L_T(\varphi)+ F(N^{\theta \varphi}) \right].
\end{equation*}
\end{theorem}

\subsection{A General Large Deviation Result}

\label{Sec:thm}

In this section, we summarize the main large deviation result of \cite{BDM09}%
. Let $\{\mathcal{G}^{\epsilon }\}_{\epsilon >0}$ be a family of measurable
maps from $\mathbb{M}$ to $\mathbb{U}$, where $\mathbb{U}$ is some Polish
space. We present below a sufficient condition for a large deviation
principle to hold for the family $Z^{\epsilon }=\mathcal{G}^{\epsilon
}(\epsilon N^{\epsilon ^{-1}})$, as $\epsilon \rightarrow 0$. Define
\begin{equation}
S^{N}=\left\{ g:\mathbb{X}_{T}\rightarrow \lbrack 0,\infty ):L_{T}(g)\leq
N\right\} .  \label{eqn:SN}
\end{equation}%
A function $g\in S^{N}$ can be identified with a measure $\nu _{T}^{g}\in
\mathbb{M}$, defined by
\begin{equation*}
\nu _{T}^{g}(A)=\int_{A}g(s,x)\nu _{T}(dsdx),\quad A\in \mathcal{B}(\mathbb{X%
}_{T}).
\end{equation*}%
This identification induces a topology on $S^{N}$ under which $S^{N}$ is a
compact space. See the Appendix for a proof of this statement.
Throughout we use this topology on $S^{N}$. Define $\mathbb{S}=\cup _{N\geq
1}S^{N}$, and let
\begin{equation*}
\mathcal{U}^{N}=\{\varphi \in \bar{\mathcal{A}}:\varphi (w)\in S^{N},\bar{%
\mathbb{P}}\ a.e.\,w\}.
\end{equation*}

The following condition will be sufficient to establish an LDP for a family $%
\{Z^{\epsilon }\}_{\epsilon >0}$ defined by $Z^{\epsilon }=\mathcal{G}%
^{\epsilon }(\epsilon N^{\epsilon ^{-1}})$. When applied to the SDE (\ref%
{eq:sdeg}) later on, $\mathcal{G}^{\epsilon }$ will be the mapping that
takes the PRM into $X^{\epsilon }$.

\begin{condition}
\label{Cond:LDP2} There exists a measurable map $\mathcal{G}^0 : \mathbb{M}%
\rightarrow \mathbb{U}$ such that the following hold.

\begin{enumerate}
\item For $N \in \mathbb{N}$, let $g_n, g \in S^N$ be such that $g_n
\rightarrow g$ as $n \rightarrow \infty$. Then
\begin{equation*}
\mathcal{G}^0\left(\nu_T^{g_n}\right)\rightarrow\mathcal{G}%
^0\left(\nu_T^{g}\right).
\end{equation*}

\item For $N\in \mathbb{N}$, let $\varphi _{\epsilon },\varphi \in \mathcal{U%
}^{N}$ be such that $\varphi _{\epsilon }$ converges in distribution to $%
\varphi $ as $\epsilon \rightarrow 0$. Then
\begin{equation*}
\mathcal{G}^{\epsilon }(\epsilon N^{\epsilon ^{-1}\varphi _{\epsilon
}})\Rightarrow \mathcal{G}^{0}\left( \nu _{T}^{\varphi }\right) .
\end{equation*}
\end{enumerate}
\end{condition}

The first condition requires continuity in the control for deterministic
controlled systems. The second condition is a law of large numbers result
for small noise controlled stochastic systems. In both cases we are allowed
to assume the controls take values in a compact set.

For $\phi \in \mathbb{U}$, define $\mathbb{S}_{\phi }=\left\{ g\in \mathbb{S}%
:\phi =\mathcal{G}^{0}(\nu _{T}^{g})\right\} $. Let $I:\mathbb{U}\rightarrow
\lbrack 0,\infty ]$ be defined by
\begin{equation}
I(\phi )=\inf_{g\in \mathbb{S}_{\phi }}\left\{ L_{T}(g)\right\} ,\quad \phi
\in \mathbb{U}.  \label{Eqn: I2}
\end{equation}%
By convention, $I(\phi )=\infty $ if $\mathbb{S}_{\phi }=\varnothing $.

The following theorem was established in \cite[Theorem 4.2]{BDM09}.

\begin{theorem}
\label{Thm:LDP01} For $\epsilon >0$, let $Z^{\epsilon }$ be defined by $%
Z^{\epsilon }=\mathcal{G}^{\epsilon }(\epsilon N^{\epsilon ^{-1}})$, and
suppose that Condition \ref{Cond:LDP2} holds. Then $I$ defined as in (\ref%
{Eqn: I2}) is a rate function on $\mathbb{U}$ and the family $\{Z^{\epsilon
}\}_{\epsilon >0}$ satisfies a large deviation principle with rate function $%
I$.
\end{theorem}

For applications, the following strengthened form of Theorem \ref{Thm:LDP01}
is useful. The proof follows by straightforward modifications; for
completeness we include a sketch in the appendix.

Let $\left\{ K_{n}\subset \mathbb{X},n=1,2,\ldots \right\} $ be an
increasing sequence of compact sets such that $\cup _{n=1}^{\infty }K_{n}=%
\mathbb{X}$. For each $n$ let
\begin{eqnarray*}
\bar{\mathcal{A}}_{b,n}&\doteq& \left\{ \varphi \in \bar{\mathcal{A}}: \mbox{ for all }(t,\omega )\in [0,T]\times \mathbb{\bar{M}}\mbox{, }n\geq \varphi
(t,x,\omega )\geq 1/n\mbox{ if }x\in K_{n}\right. \\
&& \hspace{1in} \left. \mbox{ and }\varphi (t,x,\omega )=1\mbox{ if }x\in
K_{n}^{c}\right\} ,
\end{eqnarray*}%
and let $\bar{\mathcal{A}}_{b}=\cup _{n=1}^{\infty }\bar{\mathcal{A}}_{b,n}$%
. Define $\tilde{\mathcal{U}}^N=\mathcal{U}^N\cap \bar{\mathcal{A}}_{b}$.

\begin{theorem}
\label{thm:ldp} \label{Thm:LDP2} Suppose Condition \ref{Cond:LDP2} holds
with ${\mathcal{U}}^{N}$ replaced by $\tilde{\mathcal{U}}^{N}$. Then the
conclusions of Theorem \ref{Thm:LDP01} continue to hold.
\end{theorem}

\subsection{A family of SPDEs driven by Poisson Random Measures}

In this section we introduce the basic SPDE model that will be studied in
this work. We begin by giving a precise meaning to a solution for such a
SPDE and then recall a result from \cite{Kallianpur95} which gives
sufficient conditions on the coefficients ensuring the strong existence and
pathwise uniqueness of solutions. To introduce the solution space, we start
with some basic definitions (cf. \cite{Kallianpur95}).

\begin{definition}
Let $\mathcal{E}$ be a vector space. A family of norms $\{||\cdot
||_{p}:p\in \mathbb{N}_0 \}$ on $\mathcal{E}$ is called \textbf{compatible}
if for any $p,q\in \mathbb{N}_0$, whenever $\{x_{n}\}\subseteq \mathcal{E}$
is a Cauchy sequence with respect to both $||\cdot ||_{p}$ and $||\cdot
||_{q}$, and converges to 0 with respect to one norm, then it also converges
to 0 with respect to the other norm. The family is said to be \textbf{%
increasing} if for all $x \in \mathcal{E}$, $||x ||_{p}\le ||x ||_{q}$
whenever $p\le q$.
\end{definition}

\begin{definition}
\label{def: CHNS} A separable Fr\`echet space $\Phi$ is called a \textbf{%
countable Hilbertian space} if its topology is given by an increasing
sequence $|| \cdot ||_n$, $n \in \mathbb{N}_0$, of compatible Hilbertian
norms. A countable Hilbertian space $\Phi$ is called \textbf{nuclear} if for
each $n \in \mathbb{N}_0$ there exists $m>n$ such that the canonical
injection from $\Phi_m$ into $\Phi_n$ is Hilbert-Schmidt, where $\Phi_k$,
for each $k\in \mathbb{N}_0$, is the completion of $\Phi$ with respect to $%
|| \cdot ||_k$.
\end{definition}

If $\Phi $, $\{\Phi _{n}\}_{n\in \mathbb{N}_{0}}$ are as above, then $\{\Phi
_{n}\}_{n\in \mathbb{N}_{0}}$ is a sequence of decreasing Hilbert spaces and
$\Phi =\cap _{n=0}^{\infty }\Phi _{n}$. Identify $\Phi _{0}^{\prime }$ with $%
\Phi _{0}$ using Riesz's representation theorem, and denote the space of
bounded linear functionals on $\Phi _{n}$ by $\Phi _{-n}$. This space has a
natural inner product  [and norm] which we denote by $\langle \cdot ,\cdot
\rangle _{-n}$ [resp. $||\cdot ||_{-n}$], $n\in \mathbb{N}_{0}$ such that $%
\{\Phi _{-n}\}_{n\in \mathbb{N}_{0}}$ is a sequence of increasing Hilbert
spaces and the topological dual of $\Phi$, denoted as $\Phi^{\prime} $
equals $\cup _{n=0}^{\infty }\Phi _{-n}$ (see Theorem 1.3.1 of \cite%
{Kallianpur95}). Elements of $\Phi ^{\prime }$ need not have much
regularity. Solutions of the SPDE considered in this paper will have sample
paths in $\Phi ^{\prime }$. In fact under the conditions imposed here the
solutions will take values in $D([0,T]:\Phi _{-n})$ for some finite value of
$n$.

We will assume that there is a sequence $\{\phi _{j}\}\subset \Phi $ such
that $\{\phi _{j}\}$ is a complete orthonormal system (CONS) in $\Phi _{0}$
and is a complete orthogonal system (COS) in each $\Phi _{n},n\in \mathbb{Z}$%
. Then $\{\phi _{j}^{n}\}=\{\phi _{j}||\phi _{j}||_{n}^{-1}\}$ is a CONS in $%
\Phi _{n}$ for each $n\in \mathbb{Z}$. Define the map $\theta _{p}:\Phi
_{-p}\rightarrow \Phi _{p}$ by $\theta _{p}(\phi _{j}^{-p})=\phi _{j}^{p}$.
It is easy to check that for all $p\in \mathbb{N}$, $\theta _{p}(\Phi
)\subseteq \Phi $ (see Remark 6.1.1 of \cite{Kallianpur95}). Also, for each $%
r>0$, $\eta \in \Phi _{-r}$ and $\phi \in \Phi _{r}$, $\eta \lbrack \phi ]$
is defined by the formula
\begin{equation}
\eta \lbrack \phi ]=\sum_{j=1}^{\infty }\langle \eta ,\phi _{j}\rangle
_{-r}\langle \phi ,\phi _{j}\rangle _{r}.  \label{eqn:defsqbrkt}
\end{equation}%
We refer the reader to Example 1.3.2 of \cite{Kallianpur95} for a canonical
example of such a Countable Hilbertian Nuclear Space (CHNS) defined using a
closed densely defined self-adjoint operator on $\Phi _{0}$. A similar
example is considered in Section \ref{sec:eg} of this paper.

Following \cite{kallianpur1994}, we introduce the following conditions on
the coefficients $A$ and $G$ in equation \eqref{eq:sdeg}. Let $A:[0,T]\times
\Phi ^{\prime }\rightarrow \Phi ^{\prime }$, $G:[0,T]\times \Phi ^{\prime
}\times \mathbb{X}\rightarrow \Phi ^{\prime }$ be maps satisfying the
following condition.

\begin{condition}
\label{assump:sde} There exists $p_{0}\in \mathbb{N}$ such that, for every $%
p\geq p_{0}$, there exists $q\geq p$ and a constant $K=K(p,q)$ such that the
following hold.

\begin{enumerate}
\item (Continuity) For all $t\in \lbrack 0,T]$ and $u\in \Phi _{-p}$, $%
A(t,u)\in \Phi _{-q}$ and $G(t,u,\cdot )\in L^{2}(\mathbb{X},\nu ;\Phi
_{-p}) $. The maps $u\mapsto A(t,u)$ and $u\mapsto G(t,u,\cdot )$ are
continuous.

\item (Coercivity) For all $t\in \lbrack 0,T]$, and $\phi \in \Phi $,
\begin{equation*}
2A(t,\phi )[\theta _{p}\phi ]\leq K(1+||\phi ||_{-p}^{2}).
\end{equation*}

\item (Growth) For all $t\in \lbrack 0,T]$, and $u\in \Phi _{-p}$,
\begin{equation*}
||A(t,u)||_{-q}^{2}\leq K(1+||u||_{-p}^{2})
\end{equation*}%
and
\begin{equation*}
\int_{\mathbb{X}}||G(t,u,v)||_{-p}^{2}\nu (dv)\leq K(1+||u||_{-p}^{2}).
\end{equation*}

\item (Monotonicity) For all $t\in \lbrack 0,T]$, and $u_{1},u_{2}\in \Phi
_{-p}$,
\begin{equation*}
\begin{split}
2& \langle A(t,u_{1})-A(t,u_{2}),u_{1}-u_{2}\rangle _{-q} \\
& +\int_{\mathbb{X}}||G(t,u_{1},v)-G(t,u_{2},v)||_{-q}^{2}\nu (dv)\leq
K||u_{1}-u_{2}||_{-q}^{2}.
\end{split}%
\end{equation*}
\end{enumerate}
\end{condition}

In Section \ref{sec:eg}, we will consider a model motivated by problems in
hydrology where all parts of Condition \ref{assump:sde} are satisfied.

We now give a precise definition of a solution to the SDE \eqref{eq:sdeg}.

\begin{definition}
\label{def:sde}Let $(\bar{\mathbb{M}},\mathcal{B}(\bar{\mathbb{M}}),\bar{%
\mathbb{P}},\{\bar{\mathcal{F}}_{t}\})$ be the filtered probability space
from Section \ref{Sec:PRM}. Fix $p\in \mathbb{N}_{0}$, suppose that $X_{0}$
is a $\bar{\mathcal{F}}_{0}$-measurable $\Phi _{-p}$-valued random variable
such that $\mathbb{E}||X_{0}||_{-p}^{2}<\infty $. A stochastic process $%
\{X_{t}^{\epsilon }\}_{t\in \lbrack 0,T]}$ defined on $\bar{\mathbb{M}}$ is
said to be a $\Phi _{-p}$-valued strong solution to the SDE \eqref{eq:sdeg}
with initial value $X_{0}$, if

(a) $X_{t}^{\epsilon }$ is a $\Phi _{-p}$-valued $\bar{\mathcal{F}}_{t}$%
-measurable random variable for all $t\in \lbrack 0,T]$;

(b) $X^\epsilon\in D([0,T]:\Phi_{-p} )$ a.s.;

(c) there is a $q\geq p$ such that for all $t\in \lbrack 0,T]$ and $u\in
\Phi _{-p}$, $A(t,u)\in \Phi _{-q}$ and $G(t,u,\cdot )\in L^{2}(\mathbb{X}%
,\nu ;\Phi _{-q})$, and there exists a sequence $\{\sigma _{n}\}_{n\geq 1}$
of $\{\bar{\mathcal{F}}_{t}\}$-stopping times increasing to infinity such
that for each $n\geq 1$,
\begin{equation*}
\bar{\mathbb{E}}\int_{0}^{T\wedge \sigma _{n}}\int_{\mathbb{X}%
}||G(s,X_{s}^{\epsilon },v)||_{-q}^{2}\nu (dv)ds<\infty
\end{equation*}%
and
\begin{equation*}
\bar{\mathbb{E}}\int_{0}^{T\wedge \sigma _{n}}||A(s,X_{s}^{\epsilon
})||_{-q}^{2}ds<\infty ;
\end{equation*}

(d) for all $t\in \lbrack 0,T]$, almost all $\omega \in \bar{\mathbb{M}}$,
and all $\phi \in \Phi $
\begin{equation}
X_{t}^{\epsilon }[\phi ]=X_{0}[\phi ]+\int_{0}^{t}A(s,X_{s}^{\epsilon
})[\phi ]ds+\epsilon \int_{0}^{t}\int_{\mathbb{X}}G(s,X_{s-}^{\epsilon
},v)[\phi ]\tilde{N}^{\epsilon ^{-1}}(dsdv).  \label{eq:sdeg2}
\end{equation}
\end{definition}

In Definition \ref{def:sde}, $\tilde{N}^{\epsilon ^{-1}}$ is the compensated
version of ${N}^{\epsilon ^{-1}}$ as defined below \eqref{eq:sdeg}, with ${N}%
^{\epsilon ^{-1}}$ having jump rates that are scaled by $1/\epsilon $ and is
constructed from $\bar{N}$, as below \eqref{Eqn: control}.

One can similarly define a $\Phi_{-p}$-valued strong solution on an
arbitrary filtered probability space supporting a suitable PRM.

\begin{definition}[pathwise uniqueness]
We say that the $\Phi _{-p}$-valued solution for the SDE (\ref{eq:sdeg}) has
the \textbf{pathwise uniqueness} property if the following is true. Suppose
that $X$ and $X^{\prime }$ are two $\Phi _{-p}$-valued solutions defined on
the same filtered probability space with respect to the same Poisson random
measure and starting from the same initial condition $X_{0}$. Then the paths
of $X$ and $X^{\prime }$ coincide for almost all $\omega $.
\end{definition}

The following theorem is taken from \cite{Kallianpur95} (see Theorem 6.2.2,
Lemma 6.3.1 and Theorem 6.3.1 therein).

\begin{theorem}
\label{Thm: strsol} Suppose that Condition \ref{assump:sde} holds. Let $%
X_{0} $ be a $\Phi _{-p}$-valued random variable satisfying $\mathbb{E}%
||X_{0}||_{-p}^{2}<\infty $. Then for sufficiently large $p_{1}\geq p$, the
canonical injection from $\Phi _{-p}$ to $\Phi _{-p_{1}}$ is
Hilbert-Schmidt, and for all such $p_{1}$ the SDE \eqref{eq:sdeg} with
initial value $X_{0}$ has a pathwise unique $\Phi _{-p_{1}}$-valued strong
solution.
\end{theorem}

\section{Large Deviation Principle}

\label{sec:three} Throughout this section we will assume that Condition \ref%
{assump:sde} holds.

Fix $p\geq p_{0}$ and $X_{0}\in \Phi _{-p}$. Let $X^{\epsilon }$ be the $%
\Phi _{-p_{1}}$-valued strong solution to the SDE \eqref{eq:sdeg} with
initial value $X_{0}$. In this section, we establish an LDP for $%
\{X^{\epsilon }\}$ under suitable assumptions, by verifying the sufficient
condition in Section \ref{Sec:thm}.

We begin by introducing the map $\mathcal{G}^{0}$ that will be used to
define the rate function and also used for verification of Condition \ref%
{Cond:LDP2}. Recall that $\mathbb{S}=\cup _{N\geq 1}S^{N}$, where $S^{N}$ is
defined in (\ref{eqn:SN}). As a first step we show that under Conditions \ref%
{Assump:expint2} and \ref{Assump:explip} below, for every $g\in \mathbb{S}$,
the integral equation
\begin{equation}
\tilde{X}_{t}^{g}=X_{0}+\int_{0}^{t}A(s,\tilde{X}_{s}^{g})ds+\int_{0}^{t}%
\int_{\mathbb{X}}G(s,\tilde{X}_{s}^{g},v)(g(s,v)-1)\nu (dv)ds  \label{eq:g0g}
\end{equation}%
has a unique continuous solution. Here $g$ plays the role of a control.
Keeping in mind that (\ref{eq:sdeg2}) is driven by the compensated measure
and that equations such as (\ref{eq:g0g}) will arise as law of large number
limits, $g$ corresponds to a shift in the scaled jump rate away from that of
the original model, which corresponds to $g=1$. Let
\begin{equation*}
||G(t,v)||_{0,-p}=\sup_{u\in \Phi _{-p}}\frac{||G(t,u,v)||_{-p}}{1+||u||_{-p}%
},\quad (t,v)\in \lbrack 0,T]\times \mathbb{X}.
\end{equation*}

\begin{condition}[Exponential Integrability]
\label{Assump:expint2} There exists $\delta _{1}\in(0,\infty) $ such that
for all $E\in \mathcal{B}([0,T]\times \mathbb{X})$ satisfying $\nu
_{T}(E)<\infty $,
\begin{equation*}
\int_{E}e^{\delta _{1}||G(s,v)||_{0,-p}^{2}}\nu (dv)ds<\infty .
\end{equation*}
\end{condition}

\begin{remark}
\label{rmk:exp}Under Condition \ref{Assump:expint2}, for every $\delta
_{2}\in(0,\infty) $ and for all $E\in \mathcal{B}([0,T]\times \mathbb{X})$
satisfying $\nu _{T}(E)<\infty $
\begin{equation*}
\int_{E}e^{\delta _{2}||G(s,v)||_{0,-p}}\nu (dv)ds<\infty .
\end{equation*}
\end{remark}

The proof of Remark \ref{rmk:exp} is given in the appendix.

\begin{remark}
\label{rmk:l} The following inequalities will be used several times. Proofs
are omitted.

\begin{enumerate}
\item For $a,b,\sigma \in (0,\infty )$,
\begin{equation}
ab\leq e^{\sigma a}+\frac{1}{\sigma }(b\log b-b+1)=e^{\sigma a}+\frac{1}{%
\sigma }l(b).  \label{ineq}
\end{equation}

\item For each $\beta >0$ there exists $c_{1}(\beta )>0$, such that $%
c_{1}(\beta )\rightarrow 0$ as $\beta \rightarrow \infty $ and
\begin{equation*}
|x-1|\leq c_{1}(\beta )l(x)\mbox{ whenever }|x-1|\geq \beta .
\end{equation*}

\item For each $\beta >0$ there exists $c_{2}(\beta )<\infty $, such that
\begin{equation*}
|x-1|^{2}\leq c_{2}(\beta )l(x)\mbox{ whenever }|x-1|\leq \beta .
\end{equation*}
\end{enumerate}
\end{remark}

In particular, using the inequalities we have the following lemma.

\begin{lemma}
\label{lem:Gg}Under Conditions \ref{assump:sde} (c) and \ref{Assump:expint2}%
, for every $M\in \mathbb{N}$,
\begin{equation}
\sup_{g\in S^{M}}\int_{\mathbb{X}_{T}}||G(s,v)||_{0,-p}^{2}(g(s,v)+1)\nu
(dv)ds<\infty ,  \label{eq:G02g}
\end{equation}%
\begin{equation}
\sup_{g\in S^{M}}\int_{\mathbb{X}_{T}}||G(s,v)||_{0,-p}|g(s,v)-1|\nu
(dv)ds<\infty .  \label{eq:G01g}
\end{equation}
and
\begin{equation}  \label{eq:1481}
\begin{split}
\lim_{\delta \to 0}\sup_{g\in S^{M}}\sup_{|t-s|\leq \delta }\int_{[s ,t
]\times \mathbb{X}}||G(r,v)||_{0,-p}|g(r,v)-1|\nu (dv)dr = 0.
\end{split}%
\end{equation}
\end{lemma}

\begin{proof}
First notice that under Condition \ref{assump:sde} (c), we have
\begin{equation}
\int_{\mathbb{X}_{T}}||G(s,v)||_{0,-p}^{2}\nu (dv)ds\leq KT<\infty .
\label{eq:12}
\end{equation}%
Thus we only need to prove that
\begin{equation*}
\sup_{g\in S^{M}}\int_{\mathbb{X}_{T}}||G(s,v)||_{0,-p}^{2}g(s,v)\nu
(dv)ds<\infty .
\end{equation*}%
If $E=\{(s,v):||G(s,v)||_{0,-p}\geq 1\}$, then by \eqref{eq:12} we have $\nu
_{T}(E)<\infty $. Also, from the super linear growth of the function $l$, we
can find $\kappa _{1}$, $\kappa _{2}\in (0,\infty )$ such that for all $%
x\geq \kappa _{1}$, $x\leq \kappa _{2}l(x)$. Define $F=\{(s,v):g(s,v)\geq
\kappa _{1}\}$. Then, from \eqref{ineq}
\begin{equation*}
\begin{split}
\int_{\mathbb{X}_{T}}||G(s,v)||_{0,-p}^{2}g(s,v)\nu (dv)ds&
=\int_{E}||G(s,v)||_{0,-p}^{2}g(s,v)\nu
(dv)ds+\int_{E^{c}}||G(s,v)||_{0,-p}^{2}g(s,v)\nu (dv)ds \\
& \leq \int_{E}e^{\delta _{1}||G(s,v)||_{0,-p}^{2}}\nu (dv)ds+\frac{1}{%
\delta _{1}}\int_{E}l(g(s,v))\nu (dv)ds \\
& \quad +\int_{E^{c}\cap F}\kappa _{2}l(g(s,v))\nu (dv)ds+\kappa
_{1}\int_{E^{c}\cap F^{c}}||G(s,v)||_{0,-p}^{2}\nu (dv)ds.
\end{split}%
\end{equation*}%
Combining this estimate with Condition \ref{Assump:expint2} and the
definition of $S^{M}$, we have \eqref{eq:G02g}.

We now prove \eqref{eq:G01g} and \eqref{eq:1481}. Note that
\begin{equation*}
\begin{split}
\int_{[s ,t ]\times \mathbb{X}}& ||G(r,v)||_{0,-p}|g(r,v)-1|\nu (dv)dr \\
=& \int_{([s ,t ]\times \mathbb{X})\cap E}||G(r,v)||_{0,-p}|g(r,v)-1|\nu
(dv)dr+\int_{([s ,t ]\times \mathbb{X})\cap
E^{c}}||G(r,v)||_{0,-p}|g(r,v)-1|\nu (dv)dr.
\end{split}%
\end{equation*}%
Using \eqref{ineq} twice (once with $b=$ $g$ and once with $b=1$), for any $%
M_{0}\in (0,\infty )$
\begin{equation}  \label{eq:eq2054}
\int_{([s ,t ]\times \mathbb{X})\cap E}||G(r,v)||_{0,-p}|g(r,v)-1|\nu
(dv)dr\leq 2\int_{([s ,t]\times \mathbb{X})\cap
E}e^{M_{0}||G(r,v)||_{0,-p}}\nu (dv)dr+\frac{M}{M_{0}}.
\end{equation}%
Recalling Remark \ref{rmk:l}, for any $\theta >0$ and $g \in S^M$
\begin{align}
\int_{([s ,t ]\times \mathbb{X})\cap E^{c}}&||G(r,v)||_{0,-p}|g(r,v)-1|\nu
(dv)dr\notag  \\
=& \int_{([s ,t ]\times \mathbb{X})\cap E^{c}\cap \{|g-1|\leq \theta
\}}||G(r,v)||_{0,-p}|g-1|\nu (dv)dr\notag \\
&+ \int_{([s ,t ]\times \mathbb{X})\cap E^{c}\cap \{|g-1|>\theta
\}}||G(r,v)||_{0,-p}|g-1|\nu (dv)dr\notag \\
\leq & \left( \int_{[s ,t ]\times \mathbb{X}}||G(r,v)||_{0,-p}^{2}\nu
(dv)dr\right) ^{1/2}\sqrt{c_{2}(\theta )M}+c_{1}(\theta )M.\label{eq:eq2056}
\end{align}
The inequality in \eqref{eq:G01g} now follows on setting $s=0$, $t=T$ in %
\eqref{eq:eq2054} and \eqref{eq:eq2056} and using Condition \ref{assump:sde}
(c) and Remark \ref{rmk:exp}.

Next consider \eqref{eq:1481}. Fix $\epsilon \in (0,\infty)$. Choose $M_{0}$
such that $\frac{M}{M_{0}}\leq \frac{\epsilon }{4}$. Let $\delta _{1} \in
(0, \infty)$ be such that
\begin{equation*}
2\sup_{|t -s |\leq \delta _{1}}\int_{([s ,t]\times \mathbb{X})\cap
E}e^{M_{0}||G(r,v)||_{0,-p}}\nu (dv)dr\leq \frac{\epsilon}{4}.
\end{equation*}%
Now choose $\theta \in (0,\infty)$ such that $c_{1}(\theta )M\leq \frac{%
\epsilon}{4}$. Finally, choose $\delta _{2} \in (0,\infty)$ such that
\begin{equation*}
\sup_{|t -s |\leq \delta _{2}}\left( \int_{[s ,t ]\times \mathbb{X}%
}||G(r,v)||_{0,-p}^{2}\nu (dv)dr\right) ^{1/2}\sqrt{c_{2}(\theta )N}\leq
\frac{\epsilon}{4}.
\end{equation*}%
Using the above inequalities in \eqref{eq:eq2054} and \eqref{eq:eq2056}, we
have for all  $\delta \leq \min \{\delta _{1},\delta _{2}\}$,
\begin{equation*}
\sup_{g\in S^{M}}\sup_{|t-s|\leq \delta }\int_{[s ,t ]\times \mathbb{X}%
}||G(r,v)||_{0,-p}|g(r,v)-1|\nu (dv)dr \le \epsilon
\end{equation*}
The result follows.
\end{proof}

We will need the following stronger condition on fluctuations of $G$ than
(d) of Condition \ref{assump:sde}. Let
\begin{equation*}
||G(t,v)||_{1,-q}=\sup_{{\scriptstyle u_{1},u_{2}\in\Phi_{-q},\hfill }{%
\scriptstyle u_{1}\neq u_{2}\hfill }}\frac{||G(t,u_{1},v)-G(t,u_{2},v)||_{-q}%
}{||u_{1}-u_{2}||_{-q}}.
\end{equation*}

\begin{condition}
\label{Assump:explip} For $q$ as in Condition \ref{assump:sde}, there exists
$\delta >0$ such that for all $E\in \mathcal{B}([0,T]\times \mathbb{X})$
satisfying $\nu _{T}(E)<\infty $,
\begin{equation*}
\int_{E}e^{\delta ||G(s,v)||_{1,-q}^{2}}\nu (dv)ds<\infty .
\end{equation*}
\end{condition}

\begin{remark}
\label{rmk:Gg2}Under Conditions \ref{assump:sde} (d) and \ref{Assump:explip}%
, for every $M\in \mathbb{N}$,
\begin{equation*}
\sup_{g\in S^{M}}\int_{\mathbb{X}_{T}}||G(s,v)||_{1,-q}^{2}(g(s,v)+1)\nu
(dv)ds<\infty ,
\end{equation*}%
and
\begin{equation}  \label{eq:g11g}
\sup_{g\in S^{M}}\int_{\mathbb{X}_{T}}||G(s,v)||_{1,-q}|g(s,v)-1|\nu
(dv)ds<\infty .
\end{equation}
\end{remark}

The proof of this remark is similar to that of Lemma \ref{lem:Gg}, and thus
omitted. Note that Conditions \ref{Assump:expint2} and \ref{Assump:explip}
hold trivially if $||G(s,v)||_{0,-p}$ and $||G(s,v)||_{1,-q}$ are bounded in
$(s,v)$.

Recall that $p_{1}\geq p $ is chosen such that the canonical injection from $%
\Phi _{-p}$ to $\Phi _{-p_{1}}$ is Hilbert-Schmidt.

\begin{theorem}
\label{thm:existuniq} Fix $g\in \mathbb{S}$. Suppose Conditions \ref%
{assump:sde}, \ref{Assump:expint2} and \ref{Assump:explip} hold, and that $%
X_{0}\in \Phi _{-p}$. Then there exists a unique $\tilde{X}^{g}\in
C([0,T]:\Phi _{-p_{1}})$ such that for every $\phi \in \Phi $,
\begin{equation}
\tilde{X}_{t}^{g}[\phi ]=X_{0}[\phi ]+\int_{0}^{t}A(s,\tilde{X}%
_{s}^{g})[\phi ]ds+\int_{0}^{t}\int_{\mathbb{X}}G(s,\tilde{X}%
_{s}^{g},v)[\phi ](g(s,v)-1)\nu (dv)ds.  \label{eq:g01}
\end{equation}%
Furthermore, for $N\in \mathbb{N}$, $\sup_{t\in \lbrack 0,T]}\sup_{g\in
S^{N}}||\tilde{X}_{t}^{g}||_{-p}<\infty $.
\end{theorem}

We note that in the above theorem $\tilde{X}^{g}$ is a non-random element of
$C([0,T]:\Phi _{-p_{1}})$. We can now present the main large deviations
result. Recall that for $g\in \mathbb{S}$, $\nu _{T}^{g}(dsdv)=g(s,v)\nu
(dv)ds$. Define
\begin{equation}
\mathcal{G}^{0}(\nu _{T}^{g})=\tilde{X}^{g}\mbox{ for }g\in \mathbb{S}\mbox{%
, with }\tilde{X}^{g}\mbox{ given by (\ref{eq:g01}).}  \label{eq:g0}
\end{equation}%
Let $I:D([0,T]:\Phi _{-p_{1}})\rightarrow \lbrack 0,\infty ]$ be defined as
in \eqref{Eqn: I2}.

\begin{theorem}
\label{Thm: LDPg} Suppose that Conditions \ref{assump:sde}, \ref%
{Assump:expint2} and \ref{Assump:explip} hold. Then $I$ is a rate function
on $\Phi _{-p_{1}}$, and the family $\{{X}^{\epsilon }\}_{\epsilon >0}$
satisfies a large deviation principle on $D([0,T]:\Phi _{-p_{1}})$ with rate
function $I$.
\end{theorem}

We now proceed with the proofs. In Section \ref{sec:prexun} we prove Theorem %
\ref{thm:existuniq} and in Section \ref{sec:prldp}, we present the proof of
Theorem \ref{Thm: LDPg}.

\subsection{Proof of Theorem \protect\ref{thm:existuniq}}

\label{sec:prexun} The proof of the theorem is based on the following two
lemmas. The first lemma is standard and so its proof is relegated to the
appendix. The norm $||\cdot ||$ in the lemma is the Euclidean norm in $%
\mathbb{R}^{d}$.

\begin{lemma}
\label{lem:exist} Let $a,u:[0,T]\times \mathbb{R}^{d}\rightarrow \mathbb{R}%
^{d}$ and $b:[0,T]\times \mathbb{R}^{d}\rightarrow \mathbb{R}$ be measurable
functions such that, for a.e. $s\in \lbrack 0,T]$, the maps $y\mapsto a(s,y)$%
, $y\mapsto b(s,y)$ and $y\mapsto u(s,y)$ are continuous. Further suppose
that for some $\kappa \in (0,\infty )$,
\begin{align*}
||a(s,y)||+|b(s,y)|& \leq \kappa (1+||y||),\quad \mbox{ for all $s\in
\lbrack 0,T]$, $y\in \mathbb{R}^{d}$} \\
\int_{0}^{T}\sup_{y\in \mathbb{R}^{d}}||u(s,y)||ds& \leq M<\infty .
\end{align*}%
Fix $x_{0}\in \mathbb{R}^{d}$. Then there exists $x\in C([0,T]:\mathbb{R}%
^{d})$ such that $x$ satisfies the integral equation
\begin{equation}
x(t)=x_{0}+\int_{0}^{t}a(s,x(s))ds+\int_{0}^{t}b(s,x(s))u(s,x(s))ds,
\label{eq:int}
\end{equation}%
and
\begin{equation*}
\sup_{t\in \lbrack 0,T]}||x(t)||\leq (||x_{0}||+\kappa (M+T))e^{\kappa
(M+T)}.
\end{equation*}
\end{lemma}

\begin{lemma}
\label{lem:ddim} Let $\{a^{d},g^{d}\}_{d\in \mathbb{N}}$ be a sequence of
maps, $a^{d}:[0,T]\times \mathbb{R}^{d}\rightarrow \mathbb{R}^{d}$ and $%
g^{d}:[0,T]\times \mathbb{R}^{d}\times \mathbb{X}\rightarrow \mathbb{R}^{d}$%
, such that the following hold.

\begin{enumerate}
\item For each $s\in \lbrack 0,T]$ and $y\in \mathbb{R}^{d}$, $%
g^{d}(s,y,\cdot )\in L^{2}(\mathbb{X},\nu ;\mathbb{R}^{d})$ and for each $%
s\in \lbrack 0,T]$, the maps $y\mapsto a^{d}(s,y)$ and $y\mapsto
g^{d}(s,y,\cdot )$ (from $\mathbb{R}^{d}$ to $L^{2}(\mathbb{X},\nu ;\mathbb{R%
}^{d})$) are continuous.

\item For some $\kappa \in (0,\infty )$ and all $d\in \mathbb{N}$,
\begin{equation*}
2\langle a^{d}(s,y),y\rangle \leq \kappa (1+||y||^{2}),\quad \forall
(s,y)\in \lbrack 0,T]\times \mathbb{R}^{d}
\end{equation*}%
and
\begin{equation*}
\int_{\mathbb{X}}||g^{d}(s,v)||_{0}^{2}\nu (dv)\leq \kappa ,\quad \forall
s\in \lbrack 0,T],
\end{equation*}%
where $||g^{d}(s,v)||_{0}=\sup_{y\in \mathbb{R}^{d}}\frac{||g^{d}(s,y,v)||}{%
1+||y||}$.

\item For each $d\in \mathbb{N}$, there exists $\kappa _{d}\in (0,\infty )$
with
\begin{equation*}
||a^{d}(s,y)||\leq \kappa _{d}(1+||y||),\quad \forall (s,y)\in \lbrack
0,T]\times \mathbb{R}^{d}.
\end{equation*}

\item There is a $\delta _{0}\in (0,\infty )$ such that for all $E\in
\mathcal{B}([0,T]\times \mathbb{X})$ satisfying $\nu _{T}(E)<\infty $,
\begin{equation*}
\int_{E}e^{\delta _{0}||g^{d}(s,v)||_{0}}\nu (dv)ds<\infty .
\end{equation*}
\end{enumerate}

Then for any $d\in \mathbb{N}$, $\psi \in \mathbb{S}$ and $x_{0}^{d}\in
\mathbb{R}^{d}$, the equation
\begin{equation}
{x}^{d}(t)=x_{0}^{d}+\int_{0}^{t}a^{d}(s,{x}^{d}(s))ds+\int_{0}^{t}\int_{%
\mathbb{X}}g^{d}(s,{x}^{d}(s),v)(\psi (s,v)-1)\nu (dv)ds  \label{eq:d}
\end{equation}%
has a solution ${x}^{d}\in C([0,T]:\mathbb{R}^{d})$. Suppose that $%
\sup_{d\in \mathbb{N}}||x_{0}^{d}||^{2}<\infty $. Then for every $M\in
(0,\infty )$, there exists a $\tilde{\kappa}_{M}\in (0,\infty )$ such that
\begin{equation*}
\sup_{d\in \mathbb{N}}\sup_{t\in \lbrack 0,T]}||x^{d}(t)||^{2}\leq \tilde{%
\kappa}_{M},\mbox{ whenever $\psi \in S^{M}$}.
\end{equation*}
\end{lemma}

\begin{proof}
For each $d$ fixed, equation \eqref{eq:d} is the same as \eqref{eq:int} with
the following choices of $a$, $b$ and $u$:
\begin{equation*}
a(s,y)=a^{d}(s,y),
\end{equation*}%
\begin{equation*}
b(s,y)=1+||y||,
\end{equation*}%
and
\begin{equation*}
u(s,y)=\int_{\mathbb{X}}\frac{g^{d}(s,y,v)}{1+||y||}(\psi (s,v)-1)\nu (dv).
\end{equation*}%
Thus in order to prove the existence of the solutions to \eqref{eq:d}, it
suffices to verify conditions in Lemma \ref{lem:exist}. The continuity of $a$%
, $b$ and first condition in Lemma \ref{lem:exist} are immediate. The proof
of the statement%
\begin{equation}
y\mapsto u(s,y)\mbox{ is continuous for a.e. }s\in \lbrack 0,T]
\label{u_cont}
\end{equation}%
is given in the appendix. Finally note that
\begin{equation*}
\int_{0}^{T}\sup_{y\in \mathbb{R}^{d}}||u(s,y)||ds\leq \int_{0}^{T}\int_{%
\mathbb{X}}||g^{d}(s,v)||_{0}|\psi (s,v)-1|\nu (dv)ds<\infty ,
\end{equation*}%
where the last inequality follows from conditions (b) and (d) using a
similar argument as for \eqref{eq:G01g}. Thus from Lemma \ref{lem:exist},
for each $d\in \mathbb{N}$, there exists a ${x}^{d}\in C([0,T]:\mathbb{R}%
^{d})$ satisfying \eqref{eq:d}. Next note that
\begin{equation}
\begin{split}
||{x}^{d}(t)||^{2}& =||x_{0}^{d}||^{2}+2\int_{0}^{t}\left\langle {x}%
^{d}(s),\left( a^{d}(s,{x}^{d}(s))+\int_{\mathbb{X}}g^{d}(s,{x}%
^{d}(s),v)(\psi (s,v)-1)\nu (dv)\right) \right\rangle ds \\
& \leq ||x_{0}^{d}||^{2}+2\int_{0}^{t}\left\langle {x}^{d}(s),a^{d}(s,{x}%
^{d}(s))\right\rangle ds \\
& \quad +2\int_{0}^{t}||{x}^{d}(s)||\int_{\mathbb{X}}||g^{d}(s,{x}%
^{d}(s),v)||\;|\psi (s,v)-1|\nu (dv)ds \\
& \leq ||x_{0}^{d}||^{2}+\kappa \int_{0}^{t}(1+||{x}^{d}(s)||^{2})ds \\
& \quad +2\int_{0}^{t}||{x}^{d}(s)||(1+||{x}^{d}(s)||)\int_{\mathbb{X}%
}||g^{d}(s,v)||_{0}|\psi (s,v)-1|\nu (dv)ds
\end{split}
\label{eq:14}
\end{equation}%
Let
\begin{equation*}
f^{d}(s)=\int_{\mathbb{X}}||g^{d}(s,v)||_{0}|\psi (s,v)-1|\nu (dv).
\end{equation*}%
Then as before, using (b) and (d), we have that
\begin{equation}
\sup_{\psi \in S^{M}}\sup_{d\in \mathbb{N}}\int_{0}^{T}f^{d}(s)ds<\infty .
\label{eq:15a}
\end{equation}%
Also, from \eqref{eq:14} and using that $c+c^{2}\leq 1+2c^{2}$ for $c\geq 0$,%
\begin{equation*}
||{x}^{d}(t)||^{2}\leq \left( ||x_{0}^{d}||^{2}+\kappa
T+2\int_{0}^{T}f^{d}(s)ds\right) +\int_{0}^{t}||{x}^{d}(s)||^{2}(\kappa
+4f^{d}(s))ds.
\end{equation*}%
Thus, by Gronwall's inequality
\begin{equation*}
||{x}^{d}(t)||^{2}\leq \left( ||x_{0}^{d}||^{2}+\kappa
T+2\int_{0}^{T}f^{d}(s)ds\right) e^{\kappa t+4\int_{0}^{t}f^{d}(s)ds}.
\end{equation*}%
Hence if $\sup_{d\in \mathbb{N}}||x_{0}^{d}||^{2}<\infty $, then by %
\eqref{eq:15a}%
\begin{equation*}
\sup_{\psi \in S^{M}}\sup_{d\in \mathbb{N}}\sup_{t\in \lbrack
0,T]}||x^{d}(t)||^{2}<\infty .
\end{equation*}%
The lemma follows.
\end{proof}

We are now ready to prove Theorem \ref{thm:existuniq}.

\begin{proof}[Proof of Theorem \protect\ref{thm:existuniq}]
We first argue the existence of the solutions to \eqref{eq:g01}. Let $M\in
\mathbb{N}$ be such that $g\in S^{M}$. Recall the CONS $\{\phi _{k}^{p}\}$
defined by $\phi _{k}^{p}=\phi _{k}\left\Vert \phi _{k}\right\Vert
_{p}^{-1}\in \Phi _{p}$ that was introduced below Definition \ref{def: CHNS}%
. Fix $d\in \mathbb{N}$ and let $\pi :\Phi _{-p}\rightarrow \mathbb{R}^{d}$
be the mapping given by
\begin{equation*}
\pi (u)_{k}=u[\phi _{k}^{p}],\quad k=1,2,\ldots ,d
\end{equation*}%
and denote $\pi (X_{0})$ by $x_{0}^{d}$. Define $a^{d}:[0,T]\times \mathbb{R}%
^{d}\rightarrow \mathbb{R}^{d}$ and $g^{d}:[0,T]\times \mathbb{R}^{d}\times
\mathbb{X}\rightarrow \mathbb{R}^{d}$ by
\begin{equation*}
a^{d}(s,x)_{k}=A\left( s,\sum_{j=1}^{d}x_{j}\phi _{j}^{-p}\right) [\phi
_{k}^{p}]
\end{equation*}%
and
\begin{equation*}
g^{d}(s,x,v)_{k}=G\left( s,\sum_{j=1}^{d}x_{j}\phi _{j}^{-p},v\right) [\phi
_{k}^{p}].
\end{equation*}%
It is easy to verify that $a^{d}$ and $g^{d}$ satisfy the assumptions of
Lemma \ref{lem:ddim}, and therefore there exists $x^{d}\in C([0,T]:\mathbb{R}%
^{d})$ which satisfies \eqref{eq:d} with $\psi $ replaced by $g$. Define the
$\Phi _{-p}$-valued continuous function $X^{d}$, associated with $x^{d}$, by
\begin{equation*}
X_{t}^{d}=\sum_{k=1}^{d}(x_{t}^{d})_{k}\phi _{k}^{-p}.
\end{equation*}%
Then with $\tilde{\kappa}_{M}$ as in Lemma \ref{lem:ddim}, we have
\begin{equation}
\sup_{d\in \mathbb{N}}\sup_{t\in \lbrack 0,T]}||X_{t}^{d}||_{-p}^{2}\leq
\tilde{\kappa}_{M}.  \label{eq:20}
\end{equation}

Recalling the definition of $u[\phi ]$ from (\ref{eqn:defsqbrkt}), let $%
\gamma ^{d}:\Phi ^{\prime }\rightarrow \Phi ^{\prime }$ be a mapping given
by
\begin{equation*}
\gamma ^{d}u=\sum_{k=1}^{d}u[\phi _{k}^{p}]\phi _{k}^{-p}.
\end{equation*}%
Let, for $d\in \mathbb{N}$, $A^{d}:[0,T]\times \Phi ^{\prime }\rightarrow
\Phi ^{\prime }$ and $G^{d}:[0,T]\times \Phi ^{\prime }\times \mathbb{X}%
\rightarrow \Phi ^{\prime }$ be measurable mappings given by
\begin{equation*}
A^{d}(s,u)=\gamma ^{d}A(s,\gamma ^{d}u)\quad \mbox{ and }\quad
G^{d}(s,u,v)=\gamma ^{d}G(s,\gamma ^{d}u,v).
\end{equation*}%
Then $X^{d}$ solves
\begin{equation*}
X_{t}^{d}[\phi ]=X_{0}^{d}[\phi ]+\int_{0}^{t}A^{d}(s,{X}_{s}^{d})[\phi
]ds+\int_{0}^{t}\int_{\mathbb{X}}G^{d}(s,{X}_{s}^{d},v)[\phi ](g(s,v)-1)\nu
(dv)ds,\quad \phi \in \Phi .
\end{equation*}

We now argue that for each $\phi \in \Phi $, the family $\{{X}^{d}[\phi
]\}_{d\in \mathbb{N}}$ is pre-compact in $C([0,T]:\mathbb{R})$. From %
\eqref{eq:20}, we have
\begin{equation}
\sup_{d}\sup_{t\in \lbrack 0,T]}|{X}_{t}^{d}[\phi ]|\leq \sup_{d}\sup_{t\in
\lbrack 0,T]}||{X}_{t}^{d}||_{-p}||\phi ||_{p}\leq \sqrt{\tilde{\kappa}_{M}}%
||\phi ||_{p}<\infty .  \label{eq:21}
\end{equation}%
Now we consider fluctuations of ${X}^{d}[\phi ]$. For $0\leq s\leq t\leq T$,
\begin{equation*}
\begin{split}
|{X}_{t}^{d}[\phi ]-{X}_{s}^{d}[\phi ]|\leq & \int_{s}^{t}|A^{d}(r,{X}%
_{r}^{d})[\phi ]|dr+\int_{s}^{t}\int_{\mathbb{X}}|G^{d}(r,{X}%
_{r}^{d},v)[\phi ]|\;|g(r,v)-1|\nu (dv)dr \\
\leq & \int_{s}^{t}||A^{d}(r,{X}_{r}^{d})||_{-q}||\phi
||_{q}dr+\int_{s}^{t}\int_{\mathbb{X}}||G^{d}(r,{X}_{r}^{d},v)||_{-p}||\phi
||_{p}|g(r,v)-1|\nu (dv)dr.
\end{split}%
\end{equation*}%
Also, for $(s,u)\in \lbrack 0,T]\times \Phi ^{\prime }$
\begin{equation*}
\begin{split}
||A^{d}(s,u)||_{-q}^{2}& =\left\vert \left\vert \sum_{k=1}^{d}A(s,\gamma
^{d}u)[\phi _{k}^{p}]\phi _{k}^{-p}\right\vert \right\vert _{-q}^{2} \\
& =\left\vert \left\vert \sum_{k=1}^{d}A(s,\gamma ^{d}u)[\phi _{k}^{q}]\phi
_{k}^{-q}\right\vert \right\vert _{-q}^{2} \\
& =\sum_{k=1}^{d}\left( A(s,\gamma ^{d}u)[\phi _{k}^{q}]\right) ^{2} \\
& \leq ||A(s,\gamma ^{d}u)||_{-q}^{2} \\
& \leq K\left( 1+||\gamma ^{d}u||_{-p}^{2}\right) \\
& \leq K\left( 1+||u||_{-p}^{2}\right) ,
\end{split}%
\end{equation*}%
where for the second equality we use the observation
\begin{equation*}
u[\phi _{j}^{q}]\phi _{j}^{-q}=u[\phi _{j}^{p}]\phi _{j}^{-p},\quad \forall
u\in \Phi ^{\prime },\ p,q\geq 0,
\end{equation*}%
and the last inequality follows on observing that
\begin{equation*}
||\gamma ^{d}u||_{-p}^{2}\leq ||u||_{-p}^{2},\quad \forall p\geq 0.
\end{equation*}%
Similarly,
\begin{equation*}
\begin{split}
||G^{d}(s,u,v)||_{-p}^{2}& =\left\vert \left\vert \sum_{k=1}^{d}G(s,\gamma
^{d}u,v)[\phi _{k}^{p}]\phi _{k}^{-p}\right\vert \right\vert _{-p}^{2} \\
& =\sum_{k=1}^{d}\left( G(s,\gamma ^{d}u,v)[\phi _{k}^{p}]\right) ^{2} \\
& \leq ||G(s,\gamma ^{d}u,v)||_{-p}^{2}.
\end{split}%
\end{equation*}

Combining the above estimates we have
\begin{equation*}
\begin{split}
|{X}_{t}^{d}[\phi ]-{X}_{s}^{d}[\phi ]|& \leq ||\phi ||_{q}\sqrt{K}\sqrt{1+%
\tilde{\kappa}_{M}}(t-s) \\
& \quad +||\phi ||_{p}(1+\sqrt{\tilde{\kappa}_{M}})\int_{s}^{t}\int_{\mathbb{%
X}}||G(r,v)||_{0,-p}|g(r,v)-1|\nu (dv)dr.
\end{split}%
\end{equation*}%
By Lemma \ref{lem:Gg} we now see that
\begin{equation}
\lim_{\delta \rightarrow 0}\sup_{d\in \mathbb{N}}\sup_{|t-s|\leq \delta }|{X}%
_{t}^{d}[\phi ]-{X}_{s}^{d}[\phi ]|=0.  \label{eq:22}
\end{equation}%
Combining \eqref{eq:21} and \eqref{eq:22} we now have that the family $\{{X}%
^{d}[\phi ]\}$ is pre-compact in $C([0,T]:\mathbb{R})$ for every $\phi \in
\Phi $. Combining this with \eqref{eq:20} we have that $\{{X}^{d}\}_{d\in
\mathbb{N}}$ is pre-compact in $C([0,T]:\Phi _{-p_{1}})$ (cf. Theorem 2.5.2
in \cite{Kallianpur95}). Let $\tilde{X}$ be any limit point. Then by the
dominated convergence theorem and the definitions of $A^{d}$ and $G^{d}$
(see Lemma 6.1.6 and Theorem 6.2.2 of \cite{Kallianpur95}), $\tilde{X}$
satisfies the integral equation \eqref{eq:g01}. Note that the argument also
shows that whenever $g\in S^{M}$, $\sup_{t\in \lbrack 0,T]}||\tilde{X}%
_{t}||_{-p}^{2}\leq \tilde{\kappa}_{M}$.

Next, we argue uniqueness of solutions. Suppose there are two elements $%
\tilde{X}$ and $\bar{X}$ of $C([0,T]:\Phi _{-p_{1}})$ such that both satisfy %
\eqref{eq:g01}. Then, using Condition \ref{assump:sde} (d),
\begin{equation*}
\begin{split}
||\tilde{X}_{t}-\bar{X}_{t}||_{-q}^{2}&= 2\int_{0}^{t}\langle A(s,\tilde{X}%
_{s})-A(s,\bar{X}_{s}),\tilde{X}_{s}-\bar{X}_{s}\rangle _{-q}ds \\
&\quad +2\int_{0}^{t}\int_{\mathbb{X}}\langle G(s,\tilde{X}_{s},v)-G(s,\bar{X%
}_{s},v),\tilde{X}_{s}-\bar{X}_{s}\rangle _{-q}(g(s,v)-1)\nu (dv)ds \\
&\leq K\int_{0}^{t}||\tilde{X}_{s}-\bar{X}_{s}||_{-q}^{2}ds \\
& \quad +2\int_{0}^{t}||\tilde{X}_{s}-\bar{X}_{s}||_{-q}^{2}\int_{\mathbb{X}%
}||G(s,v)||_{1,-q}|g(s,v)-1|\nu (dv)ds.
\end{split}%
\end{equation*}%
Also, by Remark \ref{rmk:Gg2},
\begin{equation*}
\int_{0}^{T}\int_{\mathbb{X}}||G(s,v)||_{1,-q}|g(s,v)-1|\nu (dv)ds<\infty .
\end{equation*}%
An application of Gronwall's inequality now shows that $||\tilde{X}_{t}-\bar{%
X}_{t}||_{-q}^{2}=0$ for all $t\in \lbrack 0,T]$. Uniqueness follows.
\end{proof}

\subsection{Proof of Theorem \protect\ref{Thm: LDPg}}

\label{sec:prldp} From Theorem \ref{Thm: strsol} and by the classical
Yamada-Watanabe argument (cf. \cite{Ikeda}), for each $\epsilon >0$, there
exists a measurable map $\mathcal{G}^{\epsilon }:\mathbb{M}\rightarrow
D([0,T]:\Phi _{-p_{1}})$ such that, for any PRM $\mathbf{n}^{\epsilon ^{-1}}$
on $[0,T]\times \mathbb{X}$ with mean measure $\epsilon ^{-1}\lambda
_{T}\otimes \nu $ given on some filtered probability space, $\mathcal{G}%
^{\epsilon }(\epsilon \mathbf{n}^{\epsilon ^{-1}})$ is the unique $\Phi
_{-p_{1}}$ valued strong solution of \eqref{eq:sdeg} (with $\tilde{N}%
^{\epsilon ^{-1}}$ replaced by $\tilde{\mathbf{n}}^{\epsilon ^{-1}}={\mathbf{%
n}}^{\epsilon ^{-1}}-\epsilon ^{-1}\lambda _{T}\otimes \nu $) with initial
value $X_{0}$, where $p_{1}$ is as in the statement of Theorem \ref{Thm:
strsol}. In particular, $X^{\epsilon }=\mathcal{G}^{\epsilon }(\epsilon
N^{\epsilon ^{-1}})$ is the strong solution of \eqref{eq:sdeg} with initial
value $X_{0}$ on $(\bar{\mathbb{M}},\mathcal{B}(\bar{\mathbb{M}}),\bar{%
\mathbb{P}},\{\bar{\mathcal{F}}_{t}\})$. In view of this observation, for
proof of Theorem \ref{Thm: LDPg}, it suffices to verify Condition \ref%
{Cond:LDP2}.

We begin with the following lemma.

\begin{lemma}
\label{lem:gng} Fix $N\in \mathbb{N}$, and let $g_{n},g\in S^{N}$ be such
that $g_{n}\rightarrow g$ as $n\rightarrow \infty $. Let $h:[0,T]\times
\mathbb{X}\rightarrow \mathbb{R}$ be a measurable function such that
\begin{equation}
\int_{\mathbb{X}_{T}}|h(s,v)|^{2}\nu _{T}(dvds)<\infty ,  \label{eq:15}
\end{equation}%
and for all $\delta _{2}\in (0,\infty )$
\begin{equation}
\int_{E}e^{\delta _{2}|h(s,v)|}\nu _{T}(dvds)<\infty ,  \label{eq:23}
\end{equation}%
for all $E\in \mathcal{B}([0,T]\times \mathbb{X})$ satisfying $\nu
_{T}(E)<\infty $. Then
\begin{equation}
\int_{\mathbb{X}_{T}}h(s,v)(g_{n}(s,v)-1)\nu _{T}(dvds)\rightarrow \int_{%
\mathbb{X}_{T}}h(s,v)(g(s,v)-1)\nu _{T}(dvds)  \label{eq:17}
\end{equation}%
as $n\rightarrow \infty $.
\end{lemma}

\begin{proof}
We first argue that given $\epsilon >0$, there exists a compact set $%
K\subset \mathbb{X}$, such that%
\begin{equation}
\sup_{n}\int_{[0,T]\times K^{c}}|h(s,v)||g_{n}(s,v)-1|\nu (dv)ds\leq
\epsilon .  \label{eq:16}
\end{equation}%
For each $\beta \in (0,\infty )$ and compact $K$ in $\mathbb{X}$, the left
side of \eqref{eq:16} can be bounded by the sum of the following two terms:
\begin{equation*}
T_{1}=\sup_{n}\int_{([0,T]\times K^{c})\cap \{|g_{n}-1|>\beta
\}}|h(s,v)||g_{n}(s,v)-1|\nu (dv)ds,
\end{equation*}%
and
\begin{equation*}
T_{2}=\sup_{n}\int_{([0,T]\times K^{c})\cap \{|g_{n}-1|\leq \beta
\}}|h(s,v)||g_{n}(s,v)-1|\nu (dv)ds.
\end{equation*}%
Consider $T_{1}$ first. Then for every $L\in (0,\infty )$
\begin{equation*}
\begin{split}
T_{1}& \leq \sup_{n}\int_{([0,T]\times K^{c})\cap \{|g_{n}-1|>\beta \}\cap
\{|h|<1\}}|h(s,v)||g_{n}(s,v)-1|\nu (dv)ds \\
& \quad +\sup_{n}\int_{([0,T]\times K^{c})\cap \{|g_{n}-1|>\beta \}\cap
\{|h|\geq 1\}}|h(s,v)||g_{n}(s,v)-1|\nu (dv)ds \\
& \leq \sup_{n}\int_{([0,T]\times K^{c})\cap \{|g_{n}-1|>\beta \}\cap
\{|h|<1\}}|g_{n}(s,v)-1|\nu (dv)ds \\
& \quad +2\int_{([0,T]\times K^{c})\cap \{|h|\geq 1\}}e^{L|h(s,v)|}\nu
(dv)ds+\frac{1}{L}\sup_{n}\int_{\mathbb{X}_{T}}l(g_{n}(s,v))\nu (dv)ds.
\end{split}%
\end{equation*}%
where the inequality uses \eqref{ineq} twice (with $b=g_{n}$ and $b=1$).
Using inequality (b) of Remark \ref{rmk:l}, the first term on the right side
above can be bounded by
\begin{equation*}
c_{1}(\beta )\sup_{n}\int_{\mathbb{X}_{T}}l(g_{n}(s,v))\nu (dv)ds\leq
c_{1}(\beta )N.
\end{equation*}%
Therefore,
\begin{equation*}
T_{1}\leq c_{1}(\beta )N+2\int_{([0,T]\times K^{c})\cap \{|h|\geq
1\}}e^{L|h(s,v)|}\nu (dv)ds+\frac{1}{L}N.
\end{equation*}

Now choose $\beta$ sufficiently large so that $c_1(\beta)N\le \epsilon/6$, $%
L $ be sufficiently large so that $N/L \le \epsilon/6$. Note that from %
\eqref{eq:15}, $\nu_T\{|h|\ge 1\}<\infty$ and so by \eqref{eq:23}, $%
\int_{|h|\ge 1}e^{L|h(s,v)|}\nu_T(dvds)<\infty$. Thus we can find a compact
set $K_1\subset\mathbb{X}$ such that
\begin{equation*}
2\int_{([0,T]\times K_1^c)\cap \{|h|\ge 1\}}e^{L|h(s,v)|}\nu_T(dv ds)\le
\epsilon/6.
\end{equation*}

With $\beta $ chosen as above, consider now the term $T_{2}$. We have, using
the Cauchy-Schwartz Inequality and inequality (c) of Remark \ref{rmk:l}, for
every compact $K$,
\begin{equation*}
\begin{split}
T_{2}^{2}\leq & \int_{[0,T]\times K^{c}}|h(s,v)|^{2}\nu (dv)ds\times
c_{2}(\beta )\sup_{n}\int_{\mathbb{X}_{T}}l(g_{n}(s,v))\nu (dv)ds \\
\leq & \int_{[0,T]\times K^{c}}|h(s,v)|^{2}\nu (dv)ds\times c_{2}(\beta )N.
\end{split}%
\end{equation*}%
By \eqref{eq:15}, we can choose a compact set $K_{2}$, such that $T_{2}\leq
\epsilon /2$ with $K$ replaced by $K_{2}$. Thus by taking $K=K_{1}\cup K_{2}$%
, we have on combining the above estimates that $T_{1}+T_{2}\leq \epsilon $.
This proves \eqref{eq:16}.

In order to prove \eqref{eq:17}, it now suffices to show that, for every
compact $K\subset \mathbb{X}$,
\begin{equation}
\int_{\lbrack 0,T]\times K}h(s,v)(g_{n}(s,v)-1)\nu _{T}(dvds)\rightarrow
\int_{\lbrack 0,T]\times K}h(s,v)(g(s,v)-1)\nu _{T}(dvds).  \label{eq:18}
\end{equation}%
Fix a compact $K\subset \mathbb{X}$. From \eqref{eq:15}, we have that $%
\int_{[0,T]\times K}|h(s,v)|\nu _{T}(dvds)<\infty $. Thus to prove %
\eqref{eq:18}, it suffices to argue
\begin{equation}
\int_{\lbrack 0,T]\times K}h(s,v)g_{n}(s,v)\nu _{T}(dvds)\rightarrow
\int_{\lbrack 0,T]\times K}h(s,v)g(s,v)\nu _{T}(dvds).  \label{eq:18a}
\end{equation}%
When $h$ is bounded, (\ref{eq:18a}) can be established using Lemma 2.8 in
\cite{bodu98}. For completeness we include the proof in Appendix. For
general $h$ (which may not be bounded), it is enough to show
\begin{equation}
\sup_{n}\int_{[0,T]\times K}|h(s,v)|1_{\left\{ |h|\geq M\right\}
}g_{n}(s,v)\nu _{T}(dvds)\rightarrow 0,  \label{eq:19}
\end{equation}%
as $M\rightarrow \infty $. We have
\begin{equation*}
\begin{split}
\sup_{n}& \int_{[0,T]\times K}|h(s,v)|1_{\left\{ |h|\geq M\right\}
}g_{n}(s,v)\nu _{T}(dvds) \\
& \leq \sup_{n}\int_{([0,T]\times K)\cap \{|h|\geq M\}}e^{L|h(s,v)|}\nu
(dv)ds+\frac{1}{L}\sup_{n}\int_{\mathbb{X}_{T}}l(g_{n}(s,v))\nu (dv)ds \\
& \leq \int_{([0,T]\times K)\cap \{|h|\geq M\}}e^{L|h(s,v)|}\nu (dv)ds+\frac{%
1}{L}N.
\end{split}%
\end{equation*}%
Given $\epsilon >0$, we can choose $L$ large enough such that $N/L\leq
\epsilon /2$. Also, since
\begin{equation*}
\int_{\lbrack 0,T]\times K}e^{L|h(s,v)|}\nu _{T}(dvds)<\infty ,
\end{equation*}%
we can choose $M_{0}$ large enough such that $\int_{([0,T]\times K)\cap
\{|h|\geq M\}}e^{L|h(s,v)|}\nu (dv)ds\leq \epsilon /2$, for all $M\geq M_{0}$%
. Thus for all $M\geq M_{0}$, $\sup_{n}\int_{[0,T]\times
K}|h(s,v)|1_{|h|\geq M}g_{n}(s,v)\nu _{T}(dvds)\leq \epsilon $. Since $%
\epsilon >0$ is arbitrary, \eqref{eq:19} follows. This proves the result.
\end{proof}

We now proceed to verify the first part of Condition \ref{Cond:LDP2}. Recall
the map $\mathcal{G}^0$ defined in \eqref{eq:g0}.

\begin{proposition}
\label{prop:g1} Fix $N \in \mathbb{N}$, and let $g_n, g \in S^N$ be such
that $g_n \rightarrow g$ as $n \rightarrow \infty$. Then
\begin{equation*}
\mathcal{G}^0\left(\nu_T^{g_n}\right)\rightarrow\mathcal{G}%
^0\left(\nu_T^{g}\right).
\end{equation*}
\end{proposition}

\begin{proof}
Let $\tilde{X}^{n}=\mathcal{G}^{0}\left( \nu _{T}^{g_{n}}\right) $. By
Theorem \ref{thm:existuniq}, there exists a constant $\tilde{\kappa}\in
(0,\infty )$ such that
\begin{equation}
\sup_{n}\sup_{t\in \lbrack 0,T]}||\tilde{X}_{t}^{n}||_{-p}\leq \tilde{\kappa}%
.  \label{eq:13}
\end{equation}%
Using similar arguments as in the proof of Theorem \ref{thm:existuniq} (cf. %
\eqref{eq:21} and \eqref{eq:22}), we have, for any $\phi \in \Phi $,%
\begin{equation*}
\sup_{n}\sup_{t\in \lbrack 0,T]}|\tilde{X}_{t}^{n}[\phi ]|<\infty .
\end{equation*}%
Also,
\begin{equation*}
\begin{split}
|\tilde{X}_{t}^{n}[\phi ]-\tilde{X}_{s}^{n}[\phi ]|& \leq ||\phi ||_{q}\sqrt{%
K}\sqrt{1+\tilde{\kappa}}(t-s) \\
& \quad +||\phi ||_{p}(1+\sqrt{\tilde{\kappa}})\int_{s}^{t}\int_{\mathbb{X}%
}||G(r,v)||_{0,-p}|g_n(r,v)-1|\nu (dv)dr.
\end{split}%
\end{equation*}%
Using \eqref{eq:1481} in Lemma \ref{lem:Gg} we now have that
\begin{equation*}
\lim_{\delta \rightarrow 0}\sup_{n}\sup_{|t-s|\leq \delta }|\tilde{X}%
_{t}^{n}[\phi ]-\tilde{X}_{s}^{n}[\phi ]|=0.
\end{equation*}%
This proves that the family $\{\tilde{X}_{t}^{n}[\phi ]\}$ is pre-compact in
$C([0,T]:\mathbb{R})$ for every $\phi \in \Phi $.

Combining this with \eqref{eq:13}, we have that $\{\tilde{X}^{n}\}_{n\in
\mathbb{N}}$ is pre-compact in $C([0,T]:\Phi _{-p_{1}})$ (see Theorem 2.5.2
in \cite{Kallianpur95}). Let $\tilde{X}$ be any limit point. An application
of the dominated convergence theorem shows that, along the convergent
subsequence,
\begin{equation}
\int_{0}^{t}A(s,\tilde{X}_{s}^{n})[\phi ]ds\rightarrow \int_{0}^{t}A(s,%
\tilde{X}_{s})[\phi ]ds  \label{eq:24}
\end{equation}%
as $n\rightarrow \infty $. Furthermore, using the convergence of $\tilde{X}%
^{n}$ to $\tilde{X}$, Condition \ref{assump:sde} (d) and \eqref{eq:g11g}, we
have that
\begin{equation}
\int_{0}^{t}\int_{\mathbb{X}}G(s,\tilde{X}_{s}^{n},v)[\phi
](g_{n}(s,v)-1)\nu (dv)ds-\int_{0}^{t}\int_{\mathbb{X}}G(s,\tilde{X}%
_{s},v)[\phi ](g_{n}(s,v)-1)\nu (dv)ds\rightarrow 0.  \label{eq:25}
\end{equation}%
Here we have used the inequality
\begin{equation*}
\left |G(s,\tilde{X}_{s}^{n},v)[\phi] - G(s,\tilde{X}_{s},v)[\phi ] \right |
\le ||G(s,v)||_{1,-q} \sup_{t \in [0,T]} ||\tilde{X}_{s}^{n} - \tilde{X}%
_{s}||_{-q}
\end{equation*}
along with inequality \eqref{eq:g11g} in Remark \ref{rmk:Gg2}.

Also, from \eqref{eq:13}, we have that for some $\kappa _{1}\in (0,\infty )$
\begin{equation*}
|G(s,\tilde{X}_{s},v)[\phi ]|\leq \kappa _{1}||G(s,v)||_{0,-p},\quad \forall
(s,v)\in \mathbb{X}_{T}.
\end{equation*}%
Combining this with Condition \ref{assump:sde} (c) and Remark \ref{rmk:exp},
we now get from Lemma \ref{lem:gng} that, as $n\rightarrow \infty $,
\begin{equation}
\int_{0}^{t}\int_{\mathbb{X}}G(s,\tilde{X}_{s},v)[\phi ](g_{n}(s,v)-1)\nu
(dv)ds\rightarrow \int_{0}^{t}\int_{\mathbb{X}}G(s,\tilde{X}_{s},v)[\phi
](g(s,v)-1)\nu (dv)ds.  \label{eq:26}
\end{equation}%
Combining \eqref{eq:24}, \eqref{eq:25} and \eqref{eq:26} we now see that $%
\tilde{X}$ must satisfy the integral equation \eqref{eq:g01} for all $\phi
\in \Phi $. In view of unique solvability of \eqref{eq:g01} (Theorem \ref%
{thm:existuniq}), it now follows that $\tilde{X}=\mathcal{G}^{0}\left( \nu
_{T}^{g}\right) $. The result follows.
\end{proof}

We now proceed to the second part of Condition \ref{Cond:LDP2}. As noted in
Theorem \ref{Thm:LDP2}, it suffices to verify this condition with $\mathcal{U%
}^M$ replaced with $\tilde{\mathcal{U}}^M$.

Recall from the beginning of this section that $X^{\epsilon }=\mathcal{G}%
^{\epsilon }(\epsilon N^{\epsilon ^{-1}})$ is the strong solution of %
\eqref{eq:sdeg} with initial value $X_{0}$ on $(\bar{\mathbb{M}},\mathcal{B}(%
\bar{\mathbb{M}}),\bar{\mathbb{P}},\{\bar{\mathcal{F}}_{t}\})$. Let $\varphi
_{\epsilon }\in \tilde{\mathcal{U}}^{M}$, define $\psi _{\epsilon
}=1/\varphi _{\epsilon }$, and recall the definitions of $\bar{N}$ and $\bar{%
\nu}_{T}$ from Section \ref{Sec:PRM}. Then it is easy to check (see Theorem
III.3.24 of \cite{JacShi}, see also Lemma 2.3 of \cite{BDM09}) that
\begin{equation*}
\mathcal{E}_{t}^{\epsilon }(\psi _{\epsilon })=\exp \left\{ \int_{\mathbb{(}%
0,t]\times \mathbb{X}\times \lbrack 0,\epsilon ^{-1}]}\log (\psi _{\epsilon
}(s,x))\bar{N}(ds\,dx\,dr)+\int_{\mathbb{(}0,t]\times \mathbb{X}\times
\lbrack 0,\epsilon ^{-1}]}\left( -\psi _{\epsilon }(s,x)+1\right) \bar{\nu}%
_{T}(ds\,dx\,dr)\right\}
\end{equation*}%
is an $\left\{ \mathcal{\bar{F}}_{t}\right\} $-martingale. Consequently
\begin{equation*}
\mathbb{Q}_{T}^{\epsilon }(G)=\int_{G}\mathcal{E}_{t}^{\epsilon }(\psi
_{\epsilon })d\mathbb{\bar{P}},\quad \mbox{ for }G\in \mathcal{B(}\mathbb{%
\bar{M}}\mathcal{\mathbb{\mathcal{)}}}
\end{equation*}%
defines a probability measure on $\mathbb{\bar{M}}$, and furthermore $%
\mathbb{\bar{P}}$ and $\mathbb{Q}_{T}^{\epsilon }$ are mutually absolutely
continuous. Also it can be verified that under $\mathbb{Q}_{T}^{\epsilon }$,
$\epsilon N^{\epsilon ^{-1}\varphi _{\epsilon }}$ has the same law as that
of $\epsilon N^{\epsilon ^{-1}}$ under $\mathbb{\bar{P}}$. Thus it follows
that $\tilde{X}^{\epsilon }=\mathcal{G}^{\epsilon }(\epsilon N^{\epsilon
^{-1}\varphi _{\epsilon }})$ is the unique solution of the following
controlled stochastic differential equation:
\begin{equation}
\tilde{X}_{t}^{\epsilon }=X_{0}+\int_{0}^{t}A(s,\tilde{X}_{s}^{\epsilon
})ds+\int_{0}^{t}\int_{\mathbb{X}}G(s,\tilde{X}_{s-}^{\epsilon },v)\left(
\epsilon N^{\epsilon ^{-1}\varphi _{\epsilon }}(dsdv)-\nu (dv)ds\right) .
\label{eq:sdec}
\end{equation}

\begin{proposition}
\label{prop:g2} 
Fix $M\in \mathbb{N}$. Let $\varphi _{\epsilon },\varphi \in \tilde{\mathcal{%
U}}^{M}$ be such that $\varphi _{\epsilon }$ converges in distribution to $%
\varphi $, under $\mathbb{\bar{P}}$, as $\epsilon \rightarrow 0$. Then $%
\mathcal{G}^{\epsilon }(\epsilon N^{\epsilon ^{-1}\varphi _{\epsilon
}})\Rightarrow \mathcal{G}^{0}\left( \nu ^{\varphi }\right) .$
\end{proposition}


\begin{proof}
If $\tilde{X}^{\epsilon }=\mathcal{G}^{\epsilon }(\epsilon N^{\epsilon
^{-1}\varphi _{\epsilon }})$, then as just noted, $\tilde{X}^{\epsilon }$ is
the unique solution of \eqref{eq:sdec}. We now show that the family $\{%
\tilde{X}^{\epsilon }\}_{\epsilon >0}$ of $D([0,T]:\Phi _{-p_{1}})$ valued
random variables is tight.

We begin by showing that for some $\epsilon _{0}\in (0,\infty )$
\begin{equation}
\sup_{0<\epsilon <\epsilon _{0}}\mathbb{E}\sup_{0\leq t\leq T}||\tilde{X}%
_{t}^{\epsilon }||_{-p}^{2}<\infty .  \label{eq:11a}
\end{equation}%
Recall that $\theta _{p}$ is defined by $\theta _{p}(\phi _{j}^{-p})=\phi
_{j}^{p}$ for the CONS $\{\phi _{j}^{-p},j\in \mathbb{Z}\}$. By It\^{o}'s
formula,
\begin{equation}
\begin{split}
||\tilde{X}_{t}^{\epsilon }||_{-p}^{2}&= ||{X_{0}}||_{-p}^{2}+2%
\int_{0}^{t}A(s,\tilde{X}_{s}^{\epsilon })[\theta _{p}\tilde{X}%
_{s}^{\epsilon }]ds+2\int_{0}^{t}\int_{\mathbb{X}}\langle G(s,\tilde{X}%
_{s}^{\epsilon },v),\tilde{X}_{s}^{\epsilon }\rangle _{-p}(\varphi
_{\epsilon }-1)\nu (dv)ds \\
& \quad +\int_{0}^{t}\int_{\mathbb{X}}\left( ||\epsilon G(s,\tilde{X}%
_{s-}^{\epsilon },v)||_{-p}^{2}+2\langle \epsilon G(s,\tilde{X}%
_{s-}^{\epsilon },v),\tilde{X}_{s-}^{\epsilon }\rangle _{-p}\right) \left(
N^{\epsilon ^{-1}\varphi _{\epsilon }}(dsdv)-\epsilon ^{-1}\varphi
_{\epsilon }\nu (dv)ds\right) \\
& \quad +\epsilon \int_{0}^{t}\int_{\mathbb{X}}||G(s,\tilde{X}_{s}^{\epsilon
},v)||_{-p}^{2}\varphi _{\epsilon }\nu (dv)ds.
\end{split}
\label{eq:ito}
\end{equation}%
For completeness we include the proof of \eqref{eq:ito} in the appendix.

For the second term in \eqref{eq:ito}, we have by Condition \ref{assump:sde}
(b) that
\begin{equation}
2\int_{0}^{t}A(s,\tilde{X}_{s}^{\epsilon })[\theta _{p}\tilde{X}%
_{s}^{\epsilon }]ds\leq K\int_{0}^{t}(1+||\tilde{X}_{s}^{\epsilon
}||_{-p}^{2})ds.
\end{equation}%
Also, using $a+a^{2}\leq 1+2a^{2}$ for $a\geq 0$
\begin{equation*}
\begin{split}
\left\vert \int_{0}^{t}\int_{\mathbb{X}}\langle G(s,\tilde{X}_{s}^{\epsilon
},v),\tilde{X}_{s}^{\epsilon }\rangle _{-p}(\varphi _{\epsilon }-1)\nu
(dv)ds\right\vert & \leq \int_{0}^{t}\int_{\mathbb{X}}\frac{||G(s,\tilde{X}%
_{s}^{\epsilon },v)||_{-p}}{1+||\tilde{X}_{s}^{\epsilon }||_{-p}}(1+||\tilde{%
X}_{s}^{\epsilon }||_{-p})||\tilde{X}_{s}^{\epsilon }||_{-p}|\varphi
_{\epsilon }-1|\nu (dv)ds \\
& \leq \int_{0}^{t}(1+2||\tilde{X}_{s}^{\epsilon }||_{-p}^{2})\left( \int_{%
\mathbb{X}}||G(s,v)||_{0,-p}|\varphi _{\epsilon }-1|\nu (dv)\right) ds \\
& \leq L_{1}+2\int_{0}^{t}||\tilde{X}_{s}^{\epsilon }||_{-p}^{2}\left( \int_{%
\mathbb{X}}||G(s,v)||_{0,-p}|\varphi _{\epsilon }-1|\nu (dv)\right) ds,
\end{split}%
\end{equation*}%
where $L_{1}=\sup_{\varphi \in S^{M}}\int_{\mathbb{X}_{T}}||G(s,v)||_{0,-p}|%
\varphi -1|\nu (dv)ds<\infty $, from \eqref{eq:G01g}.

For the last term in \eqref{eq:ito}, we have
\begin{equation*}
\begin{split}
\epsilon \int_0^t \int_\mathbb{X}|| G(s, \tilde{X}_{s}^\epsilon,v)||^2_{-p}%
\varphi_\epsilon\nu(dv)ds &=\epsilon\int_0^t \int_\mathbb{X} \frac{||G(s,
\tilde{X}_{s}^\epsilon, v)||^2_{-p}}{(1+||\tilde{X}_s^\epsilon||_{-p})^2}
(1+||\tilde{X}_s^\epsilon||_{-p})^2\varphi_\epsilon\nu(dv)ds \\
&\le 2\epsilon\int_0^t (1+||\tilde{X}_s^\epsilon||^2_{-p})\left( \int_%
\mathbb{X} ||G(s, v)||^2_{0,-p}\varphi_\epsilon\nu(dv)\right)ds \\
&\le 2\epsilon L_2+2\epsilon\int_0^t ||\tilde{X}_s^\epsilon||^2_{-p}\left(
\int_\mathbb{X} ||G(s, v)||^2_{0,-p}\varphi_\epsilon\nu(dv)\right)ds,
\end{split}%
\end{equation*}
where $L_2=\sup_{\varphi\in S^M}\int_{\mathbb{X}_T} ||G(s,
v)||^2_{0,-p}\varphi\nu(dv)ds<\infty$, from \eqref{eq:G02g}.

We split the martingale term as $M_{t}=M_{t}^{1}+M_{t}^{2}$, where
\begin{equation*}
M_{t}^{1}=\int_{0}^{t}\int_{\mathbb{X}}||\epsilon G(s,\tilde{X}%
_{s-}^{\epsilon },v)||_{-p}^{2}\left( N^{\epsilon ^{-1}\varphi _{\epsilon
}}(dsdv)-\epsilon ^{-1}\varphi _{\epsilon }\nu (dv)ds\right) ,
\end{equation*}%
and
\begin{equation*}
M_{t}^{2}=\int_{0}^{t}\int_{\mathbb{X}}2\langle \epsilon G(s,\tilde{X}%
_{s-}^{\epsilon },v),\tilde{X}_{s-}^{\epsilon }\rangle _{-p}\left(
N^{\epsilon ^{-1}\varphi _{\epsilon }}(dsdv)-\epsilon ^{-1}\varphi
_{\epsilon }\nu (dv)ds\right) .
\end{equation*}%
We now use the following Gronwall inequality:
\begin{equation*}
\mbox{If }\eta \mbox{ and $\psi \geq 0$ satisfy }\eta (s)\leq
a+\int_{0}^{s}\eta (r)\psi (r)dr\mbox{ for all }s\in \lbrack 0,t]\mbox{,
then }\eta (t)\leq ae^{\int_{0}^{t}\psi (s)ds}.
\end{equation*}%
Using this inequality, the above estimates, and Lemma \ref{lem:Gg}, we have
that for some constants $L_{3},L_{4}\in (1,\infty )$,
\begin{equation}
\sup_{0\leq s\leq t}||\tilde{X}_{s}^{\epsilon }||_{-p}^{2}\leq L_{3}\left(
L_{4}+\sup_{0\leq s\leq t}|M_{s}^{1}|+\sup_{0\leq s\leq t}|M_{s}^{2}|\right)
,  \label{eq:10}
\end{equation}%
for all $\epsilon \in (0,1)$ and $t\in \lbrack 0,T]$.

For the term $M_{t}^{1}$, we have, for $\epsilon \in (0,1)$
\begin{equation}
\begin{split}
\mathbb{E}\sup_{0\leq s\leq T}|M_{s}^{1}|& \leq \mathbb{E}\left\vert
\int_{0}^{T}\int_{\mathbb{X}}||\epsilon G(s,\tilde{X}_{s-}^{\epsilon
},v)||_{-p}^{2}N^{\epsilon ^{-1}\varphi _{\epsilon }}(dsdv)\right\vert +%
\mathbb{E}\left\vert \int_{0}^{T}\int_{\mathbb{X}}||\epsilon G(s,\tilde{X}%
_{s-}^{\epsilon },v)||_{-p}^{2}\epsilon ^{-1}\varphi _{\epsilon }\nu
(dv)ds\right\vert \\
& \leq 2\mathbb{E}\int_{0}^{T}\int_{\mathbb{X}}||\epsilon G(s,\tilde{X}%
_{s}^{\epsilon },v)||_{-p}^{2}\epsilon ^{-1}\varphi _{\epsilon }\nu (dv)ds \\
& \leq 4\epsilon \mathbb{E}\int_{0}^{T}(1+||\tilde{X}_{s}^{\epsilon
}||_{-p}^{2})\left( \int_{\mathbb{X}}||G(s,v)||_{0,-p}^{2}\varphi _{\epsilon
}\nu (dv)\right) ds \\
& \leq 4\epsilon \mathbb{E}\int_{\mathbb{X}_{T}}||G(s,v)||_{0,-p}^{2}\varphi
_{\epsilon }\nu (dv)ds+4\epsilon \mathbb{E}\sup_{0\leq s\leq T}||\tilde{X}%
_{s}^{\epsilon }||_{-p}^{2}\int_{\mathbb{X}_{T}}||G(s,v)||_{0,-p}^{2}\varphi
_{\epsilon }\nu (dv)ds \\
& \leq 4\epsilon L_{2}(1+\mathbb{E}\sup_{0\leq s\leq T}||\tilde{X}%
_{s}^{\epsilon }||_{-p}^{2}).
\end{split}
\label{eq:mt1}
\end{equation}%
Next consider the term $M_{t}^{2}$. From the Burkholder-Davis-Gundy
inequality, we have that
\begin{equation}
\begin{split}
\mathbb{E}\sup_{0\leq s\leq T}|M_{s}^{2}|& \leq 4\mathbb{E}[M^{2}]_{T}^{1/2}
\\
& \leq 4\mathbb{E}\left\{ \int_{0}^{T}\int_{\mathbb{X}}4\epsilon ^{2}\langle
G(s,\tilde{X}_{s-}^{\epsilon },v),\tilde{X}_{s-}^{\epsilon }\rangle
_{-p}^{2}N^{\epsilon ^{-1}\varphi _{\epsilon }}(dsdv)\right\} ^{1/2} \\
& \leq 4\mathbb{E}\left\{ \int_{0}^{T}\int_{\mathbb{X}}4\epsilon ^{2}||G(s,%
\tilde{X}_{s-}^{\epsilon },v)||_{-p}^{2}||\tilde{X}_{s-}^{\epsilon
}||_{-p}^{2}N^{\epsilon ^{-1}\varphi _{\epsilon }}(dsdv)\right\} ^{1/2} \\
& \leq 8\mathbb{E}\left\{ \sup_{0\leq s\leq T}||\tilde{X}_{s}^{\epsilon
}||_{-p}^{2}\int_{0}^{T}\int_{\mathbb{X}}\epsilon ^{2}||G(s,\tilde{X}%
_{s-}^{\epsilon },v)||_{-p}^{2}N^{\epsilon ^{-1}\varphi _{\epsilon
}}(dsdv)\right\} ^{1/2} \\
& \leq \frac{1}{8L_{3}}\mathbb{E}\sup_{0\leq s\leq T}||\tilde{X}%
_{s}^{\epsilon }||_{-p}^{2}+128\epsilon ^{2}L_{3}\mathbb{E}\left(
\int_{0}^{T}\int_{\mathbb{X}}||G(s,\tilde{X}_{s-}^{\epsilon
},v)||_{-p}^{2}N^{\epsilon ^{-1}\varphi _{\epsilon }}(dsdv)\right) \\
& =\frac{1}{8L_{3}}\mathbb{E}\sup_{0\leq s\leq T}||\tilde{X}_{s}^{\epsilon
}||_{-p}^{2}+128\epsilon L_{3}\mathbb{E}\left( \int_{0}^{T}\int_{\mathbb{X}%
}||G(s,\tilde{X}_{s}^{\epsilon },v)||_{-p}^{2}\varphi _{\epsilon }\nu
(dv)ds\right) \\
& \leq \frac{1}{8L_{3}}\mathbb{E}\sup_{0\leq s\leq T}||\tilde{X}%
_{s}^{\epsilon }||_{-p}^{2}+256\epsilon L_{2}L_{3}(1+\mathbb{E}\sup_{0\leq
s\leq T}||\tilde{X}_{s}^{\epsilon }||_{-p}^{2}).
\end{split}
\label{eq:mt2}
\end{equation}%
For the fifth inequality, we have used the inequality $\sqrt{ab}\leq \frac{a%
}{2}+\frac{b}{2}$ with $a=\frac{1}{32L_{3}}\sup_{0\leq s\leq T}||\tilde{X}%
_{s}^{\epsilon }||_{-p}^{2}$ and $b=32L_{3}\epsilon ^{2}\int_{0}^{T}\int_{%
\mathbb{X}}||G(s,\tilde{X}_{s-}^{\epsilon },v)||_{-p}^{2}N^{\epsilon
^{-1}\varphi _{\epsilon }}(dsdv)$. Combining \eqref{eq:10}, \eqref{eq:mt1}
and \eqref{eq:mt2} we now have
\begin{equation*}
\left(\mathbb{E}\sup_{0\leq s\leq T}||\tilde{X}_{s}^{\epsilon
}||_{-p}^{2}\right) \left(1 - 4\epsilon L_2L_3 - 256\epsilon L_2L_3^2 -\frac{%
1}{8}\right) \le L_3L_4 + 4L_2L_3 + 256L_2L_3^2.
\end{equation*}
Choose $\epsilon _{0}$ small enough so that $\max \{4\epsilon_0 L_2L_3,
256\epsilon _{0}L_{2}L_{3}^{2}\} \leq \frac{1}{8}$. Then for $\epsilon \leq
\epsilon _{0}$, we have that
\begin{equation*}
\mathbb{E}\sup_{0\leq s\leq T}||\tilde{X}_{s}^{\epsilon }||_{-p}^{2}\leq
\frac{8}{5}(L_{3}L_{4}+4L_{2}L_{3}+256L_{2}L_{3}^{2}).
\end{equation*}%
This proves \eqref{eq:11a}.

In view of the estimate in \eqref{eq:11a}, to prove tightness of $\{\tilde{X}%
^{\epsilon }\}_{\epsilon \leq \epsilon _{0}}$ in $D([0,T]:\Phi _{-p_{1}})$,
it suffices to show that for all $\phi \in \Phi $, $\{\tilde{X}^{\epsilon
}[\phi ]\}_{\epsilon \leq \epsilon _{0}}$ is tight in $D([0,T]:\mathbb{R})$.
For the rest of the proof we will only consider $\epsilon \leq \epsilon _{0}$%
, however we will suppress $\epsilon _{0}$ from the notation. Fix $\phi \in
\Phi $. Let
\begin{equation*}
C_{t}^{\epsilon }=\int_{0}^{t}A(s,\tilde{X}_{s}^{\epsilon })[\phi
]ds+\int_{0}^{t}\int_{\mathbb{X}}G(s,\tilde{X}_{s}^{\epsilon },v)[\phi
](\varphi _{\epsilon }-1)\nu (dv)ds
\end{equation*}%
and
\begin{equation*}
M_{t}^{\epsilon }=\epsilon \int_{0}^{t}\int_{\mathbb{X}}G(s,\tilde{X}%
_{s-}^{\epsilon },v)[\phi ]\tilde{N}^{\epsilon ^{-1}\varphi _{\epsilon
}}(dsdv).
\end{equation*}%
To argue tightness of $C^{\epsilon }$ in $C([0,T]:\mathbb{R})$, it suffices
to show (cf. Lemma 6.1.2 of \cite{Kallianpur95}) that for all $\tau >0$,
there exists $\delta =\delta _{\tau }>0$ such that
\begin{equation}
\sup_{0\leq \epsilon \leq \epsilon _{0}}\mathbb{P}\left( \sup_{0<\beta
-\alpha <\delta }|C_{\alpha }^{\epsilon }-C_{\beta }^{\epsilon }|>\tau
\right) <\tau .  \label{eq:28}
\end{equation}

Fix $\tau >0$. Then for arbitrary $\delta >0$,
\begin{equation}
\begin{split}
\sup_{\epsilon }& \,\mathbb{P}\left( \sup_{0<\beta -\alpha <\delta
}|C_{\alpha }^{\epsilon }-C_{\beta }^{\epsilon }|>\tau \right) \\
& =\sup_{\epsilon }\mathbb{P}\left( \sup_{0<\beta -\alpha <\delta
}\left\vert \int_{\alpha }^{\beta }A(s,\tilde{X}_{s}^{\epsilon })[\phi
]ds+\int_{\alpha }^{\beta }\int_{\mathbb{X}}G(s,\tilde{X}_{s}^{\epsilon
},v)[\phi ](\varphi _{\epsilon }-1)\nu (dv)ds\right\vert >\tau \right) \\
& \leq \sup_{\epsilon }\mathbb{P}\left( \sup_{0<\beta -\alpha <\delta
}\left\vert \int_{\alpha }^{\beta }A(s,\tilde{X}_{s}^{\epsilon })[\phi
]ds\right\vert >\frac{\tau }{2}\right) \\
& \quad +\sup_{\epsilon }\mathbb{P}\left( \sup_{0<\beta -\alpha <\delta
}\left\vert \int_{\alpha }^{\beta }\int_{\mathbb{X}}G(s,\tilde{X}%
_{s}^{\epsilon },v)[\phi ](\varphi _{\epsilon }-1)\nu (dv)ds\right\vert >%
\frac{\tau }{2}\right) \\
& \leq \sup_{\epsilon }\frac{4}{\tau ^{2}}\mathbb{E}\left( \delta
^{2}\sup_{0\leq s\leq T}\left\vert A(s,\tilde{X}_{s}^{\epsilon })[\phi
]\right\vert ^{2}\right) +\sup_{\epsilon }\frac{2}{\tau }\mathbb{E}\left(
\sup_{0<\beta -\alpha <\delta }\left\vert \int_{\alpha }^{\beta }\int_{%
\mathbb{X}}G(s,\tilde{X}_{s}^{\epsilon },v)[\phi ](\varphi _{\epsilon
}-1)\nu (dv)ds\right\vert \right) .
\end{split}
\label{eq:11}
\end{equation}%
From \eqref{eq:11a} and Condition \ref{assump:sde} (c), it follows that
\begin{equation*}
\sup_{\epsilon }\mathbb{E}\left( \sup_{0\leq s\leq T}\left\vert A(s,\tilde{X}%
_{s}^{\epsilon })[\phi ]\right\vert ^{2}\right) <\infty .
\end{equation*}%
Thus we can find $\delta _{1}>0$ such that for all $\delta \leq \delta _{1}$%
, the first term on the last line of \eqref{eq:11} is bounded by $\tau /2$.

Now we consider the second term:
\begin{equation*}
\begin{split}
& \left\vert \int_{[\alpha ,\beta ]\times \mathbb{X}}G(s,\tilde{X}%
_{s}^{\epsilon },v)[\phi ](\varphi _{\epsilon }-1)\nu (dv)ds\right\vert \\
& \quad \leq ||\phi ||_{p}\left( 1+\sup_{0\leq s\leq T}||\tilde{X}%
_{s}^{\epsilon }||_{-p}\right) \int_{[\alpha ,\beta ]\times \mathbb{X}%
}||G(s,v)||_{0,-p}|\varphi _{\epsilon }-1|\nu (dv)ds\\
& \quad \leq ||\phi ||_{p}\left( 1+\sup_{0\leq s\leq T}||\tilde{X}%
_{s}^{\epsilon }||_{-p}\right) \sup_{g\in S^M}\sup_{|t-s|\le\delta}\int_{[s ,t ]\times \mathbb{X}%
}||G(s,v)||_{0,-p}|g-1|\nu (dv)ds.
\end{split}%
\end{equation*}%
Then from \eqref{eq:1481} in Lemma \ref{lem:Gg} and \eqref{eq:11a}, we can find $\delta _{2}>0$ such that for all $\delta \leq \delta _{2}$%
, the second term on the last line of \eqref{eq:11} is bounded by $\tau /2$. By taking $\delta=\min(\delta_1,\delta_2)$, \eqref{eq:28} holds and the tightness
of $\{C^{\epsilon }\}_{\epsilon \leq \epsilon _{0}}$ follows.

Next consider $M^{\epsilon }$. We have
\begin{equation}
\begin{split}
\mathbb{E}\left\langle M^{\epsilon }\right\rangle _{T}& =\epsilon \mathbb{E}%
\int_{0}^{T}\int_{\mathbb{X}}(G(s,\tilde{X}_{s}^{\epsilon },v)[\phi
])^{2}\varphi _{\epsilon }\nu (dv)ds \\
& \leq 2\epsilon ||\phi ||_{p}(1+\mathbb{E}\sup_{0\leq s\leq T}||\tilde{X}%
_{s}^{\epsilon }||_{-p}^{2})\sup_{\varphi \in S^{M}}\int_{\mathbb{X}%
_{T}}||G(s,v)||_{0,-p}^{2}\varphi \nu (dv)ds.
\end{split}
\label{eq:}
\end{equation}%
Using Lemma \ref{lem:Gg}, we have $\mathbb{E}\sup_{0\leq s\leq
T}\left\langle M^{\epsilon }\right\rangle _{s}$ goes to 0 as $\epsilon
\rightarrow 0$. Then by Theorem 6.1.1 in \cite{Kallianpur95}, for any $\phi
\in \Phi $, the sequence of semimartingales $\tilde{X}_{t}^{\epsilon }[\phi
]=X_{0}[\phi ]+C_{t}^{\epsilon }+M_{t}^{\epsilon }$ is tight in $D([0,T]:%
\mathbb{R})$. It then follows from \eqref{eq:11a} and Theorem 2.5.2 in \cite%
{Kallianpur95} that $\{\tilde{X}^{\epsilon }\}_{\epsilon \leq \epsilon _{0}}$
is tight in $D([0,T]:\Phi _{-p_{1}})$.

Choose a subsequence along which $(\tilde{X}^{\epsilon },\varphi _{\epsilon
},M^{\epsilon })$ converges in distribution to $(\tilde{X},\tilde{\varphi}%
,0) $. Without loss of generality, we can assume the convergence is almost
sure by using the Skorokhod representation theorem. Note that $\tilde{X}%
^{\epsilon }$ satisfies the following integral equation
\begin{equation*}
\tilde{X}_{t}^{\epsilon }[\phi ]=X_{0}[\phi ]+\int_{0}^{t}A(s,\tilde{X}%
_{s}^{\epsilon })[\phi ]ds+\int_{0}^{t}\int_{\mathbb{X}}G(s,\tilde{X}%
_{s}^{\epsilon },v)[\phi ](\varphi _{\epsilon }-1)\nu (dv)ds+M^{\epsilon }.
\end{equation*}%
Along the lines of Theorem \ref{thm:existuniq} and Proposition \ref{prop:g1}
(see (\ref{eq:24}) -- (\ref{eq:26})), we see that $\tilde{X}$ must solve
\begin{equation*}
\tilde{X}_{t}[\phi ]=X_{0}[\phi ]+\int_{0}^{t}A(s,\tilde{X}_{s})[\phi
]ds+\int_{0}^{t}\int_{\mathbb{X}}G(s,\tilde{X}_{s},v)[\phi ](\tilde{\varphi}%
-1)\nu (dv)ds.
\end{equation*}%
The unique solvability of the above integral equation gives that $\tilde{X}=%
\mathcal{G}^{0}\left( \nu ^{\tilde{\varphi}}\right) $, thus we have proved
part 2 of Condition \ref{Cond:LDP2}, i.e., $\mathcal{G}^{\epsilon }(\epsilon
N^{\epsilon ^{-1}\varphi _{\epsilon }})\Rightarrow \mathcal{G}^{0}\left( \nu
^{\varphi }\right) $.
\end{proof}

We are now ready to prove the main theorem.

\begin{proof}[Proof of Theorem \protect\ref{Thm: LDPg}]
Using Propositions \ref{prop:g1} and \ref{prop:g2}, Theorem \ref{Thm: LDPg}
is an immediate consequence of Theorem \ref{thm:ldp}.
\end{proof}


\section{A one dimensional model for spread of a chemical agent}

\label{sec:eg} In the hydrology literature (see \cite{Tamai1976} for
example), partial differential equations of the following type are often
used to model the spread of a pollutant in a reservoir, river or air:
\begin{equation}
D\triangle \phi -V\cdot \nabla \phi -\alpha \phi +Q=0.  \label{ins258}
\end{equation}%
Here $\phi (x)$ represents the water quality or pollutant concentration at
location $x$; $\triangle $ is the Laplacian operator modeling the diffusion
of the chemical; $D$ is the coefficient capturing the strength of the
diffusion effect. The term $V\cdot \nabla \phi $ models the convection term,
here $\nabla $ is the gradient operator and $V$ is the velocity vector. The
scalar $\alpha \geq 0$ can be interpreted as the rate of dissipation of the
chemical and $Q\geq 0$ is the \textquotedblleft load\textquotedblright\ or
pollutant issued from outside. Pollutants take various forms, such as
nutrients (e.g., runoff fertilizer), microbiological, and chemical (e.g.,
pesticides).

The deterministic equation (\ref{ins258}) models the steady state density
profile of the pollutant and does not take into account any temporal or
stochastic variability. A dynamic stochastic model for pollutant spread
described through a stochastic partial differential equation (SPDE) driven
by a PRM was studied in \cite{Kallianpur95}. We begin by describing this
model in a one dimensional setting, where it describes the evolution of a
pollutant deposited at different sites along a reservoir. Our goal is to
study probabilities of deviations from the nominal behavior by establishing
a suitable large deviation principle.

\subsection{Dynamic SPDE Model}

\label{sec:model} The model considered here describes the spread of a
chemical agent which is released by several different sources along a
one-dimensional reservoir. Suppose that there are $r$ such sources located
at different sites $\kappa _{1},\ldots ,\kappa _{r}\in \lbrack 0,l]$, where
the interval $[0,l]$ represents the reservoir. These sources release
pollutants according to independent Poisson streams $N_{i}(t)$, with rate $%
f_{i}$, $i=1,...,r$, and with random magnitudes $A_{i}^{j}(\omega )$, $j\in
\mathbb{N}$, $i=1,...,r$, which are mutually independent with magnitudes in
the $i^{th}$ stream having common distribution $F_{i}(da)$.

Formally, the model describing the evolution of concentration is written as
follows:
\begin{equation}
\begin{split}
\frac{\partial }{\partial t}u(t,x)& =D\frac{\partial ^{2}}{\partial x^{2}}%
u(t,x)-V\frac{\partial }{\partial x}u(t,x)-\alpha u(t,x) \\
& \quad +\sum_{i=1}^{r}\sum_{j}A_{i}^{j}(\omega )\delta _{\kappa
_{i}}(x)1_{\left\{ t=\tau _{i}^{j}(\omega )\right\} }
\end{split}
\label{eq:model}
\end{equation}%
where $\tau _{i}^{j}(\omega )$, $j\in \mathbb{N}$ are the jump times of $%
N_{i}$, and $\delta _{a}(x)$ is the Dirac delta measure with unit mass at $a$%
. The equation is considered with a Neumann boundary condition on $[0,l]$. A
Neumann boundary condition is reasonable as a model for a reservoir, though
one would expect in this case that at the boundary the component of the
velocity orthogonal to the boundary would be zero, which in the current
setting would mean $V=0$. However, the example is for illustrative purposes
only, and the domain, boundary conditions and differential operator may be
made much more general, though one will not always obtain expressions as
explicit as those given below.

The equation (\ref{eq:model}) can be regarded as a stochastic partial
differential equation driven by a Poisson random measure. The Poisson random
measure $N$ driving the equation is a random measure on the space $\mathbb{R}%
_{+}\times \mathbb{X}$ with $\mathbb{X}=\mathbb{J}\times \mathbb{R}_{+}$ and
$\mathbb{J}=\{1,2,...,r\}$, and can be represented as
\begin{equation*}
N([0,t]\times A\times
B)=\sum_{i=1}^{r}1_{A}(i)\sum_{j=1}^{N_{i}(t)}1_{B}(A_{i}^{j}(\omega
)),\quad t\geq 0,\ A\subseteq \mathbb{J},\ B\in \mathcal{B}(\mathbb{R}_{+}).
\end{equation*}%
The intensity measure of $N$ is given by $\nu _{0}=\lambda \otimes \nu $,
where $\lambda $ is the Lebesgue measure on $\mathbb{R}_{+}$ and
\begin{equation}
\nu (A\times B)=\sum_{i=1}^{r}1_{A}(i)f_{i}F_{i}(B),\quad \ A\subseteq
\mathbb{J},\ B\in \mathcal{B}(\mathbb{R}_{+}).  \label{intensity}
\end{equation}

We now introduce a natural CHNS associated with equation (\ref{eq:model})
(see \cite{Kallianpur95}). Let $\rho \in \mathcal{M}_{F}[0,l]$ be defined by
\begin{equation*}
\rho (A)=\int_{A}e^{-2cx}dx;\quad A\in \mathcal{B}[0,l],
\end{equation*}%
where $c=\frac{V}{2D}$. Let $H=L^{2}([0,l],\rho )$. Then $\{\phi
_{j}\}_{j\in \mathbb{N}_{0}}$ defined below is a complete orthonormal system
on $H$ of eigen-functions of the operator $L$ defined by
\begin{equation}
L\phi =D\frac{\partial ^{2}}{\partial x^{2}}\phi -V\frac{\partial }{\partial
x}\phi ,  \label{eq:L}
\end{equation}%
with Neumann boundary $\phi ^{\prime }(0)=\phi ^{\prime }(l)=0$.
\begin{equation*}
\phi _{0}(x)=\sqrt{\frac{2c}{1-e^{-2cl}}},\quad \phi _{j}(x)=\sqrt{\frac{2}{l%
}}e^{cx}\sin \left( \frac{j\pi }{l}x+\alpha _{j}\right) ;
\end{equation*}%
\begin{equation*}
\alpha _{j}=\tan ^{-1}(-\frac{j\pi }{lc}),\quad j=1,2,\ldots .
\end{equation*}%
The corresponding eigenvalues, denoted by $\{-\lambda _{j}\}_{j\in \mathbb{N}%
_{0}}$, are given as
\begin{equation*}
\lambda _{0}=0,\quad \lambda _{j}=D\left( c^{2}+\left( \frac{j\pi }{l}%
\right) ^{2}\right) .
\end{equation*}%
For $\phi \in H$ and $n\in \mathbb{Z}$ let
\begin{equation*}
||\phi ||_{n}^{2}=\sum_{j=0}^{\infty }\langle \phi ,\phi _{j}\rangle
^{2}(1+\lambda _{j})^{2n},
\end{equation*}%
where $\langle \phi ,\psi \rangle $ is the inner product on $H$. Define
\begin{equation}
\Phi =\{\phi \in H:||\phi ||_{n}<\infty ,\forall n\in \mathbb{Z}\}
\label{eq:defofPhi}
\end{equation}%
and let $\Phi _{n}$ be the completion of $\Phi $ with respect to the norm $%
||\cdot ||_{n}$. Note that $\Phi _{0}=H$, and it can be checked that $\Phi $
is a CHNS.

With $\Phi $ defined by (\ref{eq:defofPhi}), the equation in \eqref{eq:model}
can be written rigorously as a SPDE in $\Phi ^{\prime }$ as follows. Define $%
A:\Phi ^{\prime }\rightarrow \Phi ^{\prime }$ and $G:\mathbb{X}\rightarrow
\Phi ^{\prime }$ by
\begin{equation}
A(u)[\phi ]=u[L\phi ]-\alpha u[\phi ]+\sum_{i=1}^{r}a_{i}f_{i}\phi (\kappa
_{i})\rho (\kappa _{i}),\quad \phi \in \Phi ,u\in \Phi ^{\prime }
\label{eq:A}
\end{equation}%
\begin{equation}
G(i,a)[\phi ]=a\phi (\kappa _{i})\rho (\kappa _{i}),\;(i,a)\in \mathbb{J}%
\times \mathbb{R}_{+},\ \phi \in \Phi  \label{eq:G}
\end{equation}%
where $a_{i}=\int_{\mathbb{R}_{+}}aF_{i}(da)$ and $L$ is defined as in (\ref%
{eq:L}).

Let $(\Omega ,\mathscr{F},\mathbb{P},\{\mathcal{F}_{t}\})$ be a filtered
probability space on which is given a Poison random measure $N$ with
intensity measure $\lambda \otimes \nu $, with $\nu $ as in \eqref{intensity}%
, such that $N([0,t]\times A\times B)-t\nu (A\times B)$ is a $\{\mathcal{F}%
_{t}\}$ martingale for all $A\subseteq \mathbb{J},\ B\in \mathcal{B}(\mathbb{%
R}_{+})$ satisfying $\nu (A\times B)<\infty $, and let $u_{0}$ be a $%
\mathcal{F}_{0}$-measurable random variable with values in $\Phi ^{\prime }$%
. In order to formulate the SPDE, we will need square integrability
assumptions on $F_{i}$, but with large deviations questions in mind, we
impose the following stronger integrability requirement.

\begin{condition}
\label{Assump:expint} There exists $\delta >0$ such that
\begin{equation*}
\int_0^{\infty} e^{\delta a^2}F_i(da) < \infty, \ \ \ \ \forall i = 1,...,r.
\end{equation*}
\end{condition}

Let $\tilde{N}(ds dv)$ be the compensated random measure of $N$, i.e.
\begin{equation*}
\tilde{N}([0,t] \times B)= N([0,t] \times B) - t \nu(B),
\end{equation*}
$\forall B \in \mathcal{B}(\mathbb{X})$ with $\nu(B) < \infty$. Note that
the operator $-L$ on $H$ is positive definite and self-adjoint, and thus the
following definition of a solution of \eqref{eq:model} is natural.

\begin{definition}
Fix $p\geq 0$, suppose that $\mathbb{E}||u_{0}||_{-p}^{2}<\infty $. A
stochastic process $\{u_{t}\}_{t\in \lbrack 0,\infty )}$ defined on $(\Omega
,\mathscr{F},\mathbb{P})$ is said to be a $\Phi _{-p}$-valued strong
solution to the SPDE \eqref{eq:model} with initial value $u_{0}$, if

(a) $u_{t}$ is a $\Phi _{-p}$-valued ${\mathcal{F}}_{t}$-measurable random
variable, for all $t\in \lbrack 0,\infty )$;

(b) $u\in D([0,\infty): \Phi_{-p})$ a.s.;

(c) For all $t \in [0,\infty)$ and a.e. $\omega$
\begin{equation*}
u_t[\phi] = u_0[\phi] + \int_0^t A(u_s)[\phi] ds + \int_0^t \int_\mathbb{X}
G(v)[\phi] \tilde{N}(ds dv), \quad \forall \phi \in \Phi.
\end{equation*}
\end{definition}

We are interested in the behavior of the solution when the Poisson noise is
small, namely the case where $\nu _{0}$ is replaced with $\epsilon ^{-1}\nu
_{0}$ and $G$ with $\epsilon G$, and $\epsilon $ is a small parameter. More
precisely, the goal is to study the large deviation behavior of $%
\{u_{t}^{\epsilon }\}_{0\leq t\leq T}$ in $D([0,T]:\Phi _{-p})$, as $%
\epsilon \rightarrow 0$, where $u^{\epsilon }$ solves the integral equation
\begin{equation*}
u_{t}^{\epsilon }=u_{0}+\int_{0}^{t}A(u_{s}^{\epsilon })ds+\epsilon
\int_{0}^{t}\int_{\mathbb{X}}G(v)\tilde{N}^{\epsilon ^{-1}}(dsdv),
\end{equation*}%
where $\tilde{N}^{\epsilon ^{-1}}$ is the compensated version of ${N}%
^{\epsilon ^{-1}}$ as introduced below \eqref{eq:sdeg} and ${N}^{\epsilon
^{-1}}$ is constructed using $\bar{N}$ as in \eqref{Eqn: control}. Here $%
\bar{N}$, as in Section \ref{prelim}, is once more a Poisson random measure
on $[0,T]\times \mathbb{X}\times \lbrack 0,\infty )$ with intensity $\bar{\nu%
}_{T}=\lambda _{T}\otimes \nu \otimes \lambda _{\infty }$. In particular, we
are assuming (without loss of generality) that the filtered probability
space $(\Omega ,\mathscr{F},\mathbb{P},\{\mathcal{F}_{t}\})$ introduced
below \eqref{eq:G} is large enough to support the Poisson random measure $%
\bar{N}$ that has the usual martingale properties with respect to the
filtration $\{\mathcal{F}_{t}\}$.

It can be easily checked that the functions $A$ and $G$ satisfy Condition %
\ref{assump:sde} with $p_{0}=1$. Moreover in the setting of this section,
for any $p_{1}\geq 2$, the canonical injection from $\Phi _{-1}$ to $\Phi
_{-p_{1}}$ is Hilbert-Schmidt. Recall the space $\mathbb{M}$ and $\mathbb{S}$
from Section \ref{prelim}. Recall that $\nu _{T}^{g}(dsdv)=g(s,v)\nu (dv)ds$.

For $p_{1}\geq 2$ fixed, define the map $\mathcal{G}^{0}:\mathbb{M}%
\rightarrow \mathbb{U}=D([0,T]:\Phi _{-p_{1}})$ as follows.
\begin{equation*}
\mathcal{G}^{0}(\nu _{T}^{g})=\tilde{u}^{g}\mbox{ for }g\in \mathbb{S}\mbox{%
, with }\tilde{u}^{g}\mbox{ given by (\ref{geqn})}.
\end{equation*}%
%
%
%
%
%
%
%
%
%
\begin{equation}
\tilde{u}_{t}^{g}[\phi ]=u_{0}[\phi ]+\int_{0}^{t}A(\tilde{u}_{s}^{g})[\phi
]ds+\int_{0}^{t}\int_{\mathbb{X}}G(v)[\phi ](g(s,v)-1)\nu _{T}(dvds),\quad
\forall \phi \in \Phi .  \label{geqn}
\end{equation}%
From Theorem \ref{thm:existuniq}, we have that there is a unique $\tilde{u}%
^{g}\in D([0,T]:\Phi _{-p_{1}})$ that solves (\ref{geqn}).

Define $I$ through (\ref{Eqn: I2}), where $L_T$ is as in \eqref{Ltdef}. It
can be checked that Conditions \ref{Assump:expint2} and \ref{Assump:explip}
are satisfied under Condition \ref{Assump:expint}. Thus, as an immediate
consequence of Theorem \ref{Thm: LDPg}, we have the following large
deviation principle for $u^{\epsilon}$.

\begin{theorem}
\label{Thm: LDP2} Suppose Condition \ref{Assump:expint} holds. Fix $p_1 \ge
2 $. Then $I$ is a rate function on $\mathbb{U}$ and the family $\{{u}%
^{\epsilon}\}_{\epsilon>0}$ satisfies a large deviation principle, as $%
\epsilon\rightarrow 0 $, on $D([0,T]:\Phi_{-p_1})$, with rate function $I$.
\end{theorem}

Note that as $\epsilon \rightarrow 0$, $u^{\epsilon }$ converges in $%
D([0,T]:\Phi _{-p_{1}})$ to $u^{0}$ that solves the integral equation
\begin{equation*}
u_{t}^{0}[\phi ]=u_{0}[\phi ]+\int_{0}^{t}A({u}_{s}^{0})[\phi ]ds,\quad
\forall \phi \in \Phi .
\end{equation*}%
In particular, if $u_{0}$ solves the stationary equation
\begin{equation}
D\frac{d^{2}u_{0}(x)}{dx^{2}}-V\frac{du_{0}(x)}{dx}-\alpha u_{0}(x)+Q(x)=0,
\label{eq:steady}
\end{equation}%
where
\begin{equation*}
Q(x)=\sum_{i=1}^{r}a_{i}f_{i}\delta _{\kappa _{i}}(x),
\end{equation*}%
then $u_{t}^{0}=u_{0}$ for all $t\geq 0$. It is easily verified that there
is a unique $\Phi _{-1}$ valued solution to \eqref{eq:steady} which can be
explicitly characterized by
\begin{equation*}
u_{0}[\phi ]=\sum_{i=1}^{r}\sum_{j=1}^{\infty }\frac{a_{i}f_{i}}{\alpha
+\lambda _{j}}\langle \phi ,\phi _{j}\rangle \phi _{j}(\kappa _{i})\rho
(\kappa _{i}),\quad \forall \phi \in \Phi .
\end{equation*}%
Equation (\ref{eq:steady}) should be compared with the stationary
deterministic equation (\ref{ins258}). This equation, which appears in \cite%
{Tamai1976}, has been proposed as a model for the long time concentration
profile when there is a constant rate, non random, source term given by $%
Q(x) $. Theorem \ref{Thm: LDP2} provides probabilities of large deviations
from the steady state nominal values given by (\ref{ins258}) when the true
source term is a small noise perturbation of $Q$. We remark that in this
case the solution to the integral equation for $\tilde{u}^{g}$ (i.e., %
\eqref{geqn}) that is used to define the map $\mathcal{G}^{0}$ appearing in
the formula for the rate function, can be explicitly written as
\begin{equation*}
\tilde{u}_{t}^{g}[\phi ]=\sum_{j=0}^{\infty }\sum_{i=1}^{r}e^{-(\alpha
+\lambda _{j})t}f_{i}\phi _{j}(\kappa _{i})\rho (\kappa _{i})\langle \phi
,\phi _{j}\rangle \left[ \int_{0}^{t}\int_{0}^{\infty }e^{(\alpha +\lambda
_{j})s}ag(s,i,a)F_{i}(da)ds+\frac{a_{i}}{\alpha +\lambda _{j}}\right] .
\end{equation*}

\section{Appendix}

\subsection{Proof of compactness of $S^{N}$}

\begin{lemma}
\label{lem:cmpt} For every $N \in \mathbb{N}$, $\{\nu^g_T: g \in S^N\}$ is a
compact subset of $\mathbb{M}$.
\end{lemma}

\begin{proof}
The topology on $\mathbb{M}$, which was described in Section \ref{Sec:PRM},
can be metrized as follows. Consider a sequence of open sets $\left\{
O_{j},j\in \mathbb{N}\right\} $ such that $\bar{O}_{j}\subset O_{j+1}$, each
$\bar{O}_{j}$ is compact, and $\cup _{j=1}^{\infty }O_{j}=\mathbb{X}_{T}$
(cf. Theorem 9.5.21 of \cite{Roy}). Let $\phi _{j}(x)=\left[ 1-d(x,O_{j})%
\right] \vee 0$, where $d$ denotes the metric on $\mathbb{X}_{T}$. Given any
$\mu \in \mathbb{M}$, let $\mu ^{(j)}\in \mathbb{M}$ be defined by $\left[
d\mu ^{(j)}/d\mu \right] (x)=\phi _{j}(x)$. Given $\mu ,\nu \in \mathbb{M}$,
let
\begin{equation*}
\bar{d}(\mu ,\nu )=\sum_{j=1}^{\infty }2^{-j}\left\Vert \mu ^{(j)}-\nu
^{(j)}\right\Vert _{BL},
\end{equation*}%
where $\left\Vert \cdot \right\Vert _{BL}$ denotes the bounded, Lipschitz
norm on $\mathcal{M}_{F}(\mathbb{X}_{T})$:
\begin{equation*}
\left\Vert \mu ^{(j)}-\nu ^{(j)}\right\Vert _{BL}=\sup \left\{ \int_{\mathbb{%
X}_{T}}fd\mu ^{(j)}-\int_{\mathbb{X}_{T}}fd\nu ^{(j)}:\left\vert
f\right\vert _{\infty }\leq 1,\left\vert f(x)-f(y)\right\vert \leq d(x,y)%
\mbox{ for all }x,y\in \mathbb{X}_{T}\right\} .
\end{equation*}

It is straightforward to check that $\bar{d}(\mu ,\nu )$ defines a metric
under which $\mathbb{M}$ is a Polish space, and that convergence in this
metric is essentially equivalent to weak convergence on each compact subset
of $\mathbb{X}_T$. Specifically, $\bar{d}(\mu _{n},\mu )\rightarrow 0$ if
and only if for each $j\in \mathbb{N}$, $\mu _{n}^{(j)}\rightarrow \mu
^{(j)} $ in the weak topology as finite nonnegative measures, i.e., for all $%
f\in C_{b}(\mathbb{X}_T)$%
\begin{equation*}
\int_{\mathbb{X}_T}fd\mu _{n}^{(j)}\rightarrow \int_{\mathbb{X}_T}fd\mu
^{(j)}.
\end{equation*}

Let $\mu _{n}=\nu _{T}^{g_{n}}$. We first show that $\{\mu _{n}\}\subset
\mathbb{M}$ is relatively compact for any sequence $\{g_{n}\}\subset S^{N}$.
For this, by using a diagonalization method, it suffices to show that $\{\mu
_{n}^{(j)}\}\subset \mathbb{M}$ is relatively compact for every $j$. Next,
since $\mu _{n}^{(j)}$ are supported on the compact subset of $\mathbb{X}%
_{T} $ given by $K^{j}=\overline{\{x|\phi _{j}(x)\neq 0\}}$, to show $\{\mu
_{n}^{(j)}\}\subset \mathbb{M}$ is relatively compact it suffices to show $%
\sup_{n}\mu _{n}^{(j)}(\mathbb{X}_{T})<\infty $. The last property will
follow from the fact that $L_{T}(g_{n})\leq N$ for all $n$, and the
super-linear growth of $l$. Specifically, let $c\in (0,\infty )$ be such
that $z\leq c(l(z)+1)$ for all $z\in \lbrack 0,\infty )$. Then
\begin{equation*}
\sup_{n}\mu _{n}^{(j)}(\mathbb{X}_{T})=\sup_{n}\int_{\mathbb{X}_{T}}\phi
_{j}(x)g_{n}(x)\nu _{T}(dx)\leq \sup_{n}\int_{K^{j}}g_{n}(x)\nu _{T}(dx)\leq
c(N+\nu _{T}(K^{j}))<\infty .
\end{equation*}

Next, suppose that along a subsequence (without loss of generality, also
denoted by $\{\mu _{n}\}$), $\mu _{n}\rightarrow \mu $. We would like to
show that $\mu $ is of the form $\nu _{T}^{g}$, where $g\in S^{N}$. For this
we will use the lower semi-continuity property of relative entropy. The
result holds trivially if $\mu =0$. Suppose now $\mu \neq 0$. Then there
exists $j_{0}\in \mathbb{N}$ such that for all $j\geq j_{0}$, $\inf_{n\in
\mathbb{N}}\nu _{T}^{g_{n}}(\bar{O}_{j})>0$. For $j\geq j_{0}$, define
\begin{equation*}
c^{j}=\nu _{T}^{(j)}(\mathbb{X}_{T}),\quad \bar{\nu}_{T}^{j}={\nu _{T}^{(j)}}%
/{c^{j}};
\end{equation*}%
\begin{equation*}
c_{n}^{j}=\mu _{n}^{(j)}(\mathbb{X}_{T}),\quad \bar{\mu}_{n}^{j}={\mu
_{n}^{(j)}}/{c_{n}^{j}};
\end{equation*}%
\begin{equation*}
c_{\mu }^{j}=\mu ^{(j)}(\mathbb{X}_{T}),\quad \bar{\mu}^{j}={\mu ^{(j)}}/{%
c_{\mu }^{j}}.
\end{equation*}%
Then $\bar{\nu}_{T}^{j}$, $\bar{\mu}_{n}^{j}$ and $\bar{\mu}^{j}$ are
probability measures, and
\begin{equation*}
\begin{split}
R(\bar{\mu}_{n}^{j}||\bar{\nu}_{T}^{j})& =\frac{1}{c_{n}^{j}}\int_{\mathbb{X}%
_{T}}\left[ \log (g_{n}(x))+\log \left( \frac{c^{j}}{c_{n}^{j}}\right) %
\right] g_{n}(x)\phi _{j}(x)\nu _{T}(dx) \\
& =\frac{1}{c_{n}^{j}}\int_{\mathbb{X}_{T}}\left[ l(g_{n}(x))+g_{n}(x)-1%
\right] \phi _{j}(x)\nu _{T}(dx)+\log \left( \frac{c^{j}}{c_{n}^{j}}\right)
\\
& \leq \frac{1}{c_{n}^{j}}N+1-\frac{c^{j}}{c_{n}^{j}}+\log \left( \frac{c^{j}%
}{c_{n}^{j}}\right) .
\end{split}%
\end{equation*}%
Since $\mu _{n}^{(j)}\rightarrow \mu ^{(j)}$, we have $c_{n}^{j}\rightarrow
c_{\mu }^{j}$. Thus by the lower semi-continuity property of relative
entropy,
\begin{align}
R(\bar{\mu}^{j}||\bar{\nu}_{T}^{j})& \leq \liminf_{n\rightarrow \infty }R(%
\bar{\mu}_{n}^{j}||\bar{\nu}_{T}^{j})  \notag \\
& \leq \liminf_{n\rightarrow \infty }\left[ \frac{1}{c_{n}^{j}}N+1-\frac{%
c^{j}}{c_{n}^{j}}+\log \left( \frac{c^{j}}{c_{n}^{j}}\right) \right]  \notag
\\
& \leq \frac{1}{c_{\mu }^{j}}N+1-\frac{c^{j}}{c_{\mu }^{j}}+\log \left(
\frac{c^{j}}{c_{\mu }^{j}}\right)  \label{Eqn: re} \\
& <\infty .  \notag
\end{align}%
Thus $\mu ^{(j)}$ is absolutely continuous with respect to $\nu _{T}^{(j)}$.
Define $g^{j}=d\mu ^{(j)}/d\nu _{T}^{(j)}$, and $g=g^{j}\ \mbox{on}\ \bar{O}%
_{j}$. It is easily checked that $g$ is defined consistently, and that $\mu
=\nu _{T}^{g}$. Also by a direct calculation,
\begin{equation*}
R(\bar{\mu}^{j}||\bar{\nu}_{T}^{j})=\frac{1}{c_{\mu }^{j}}\int_{\mathbb{X}%
_{T}}l(g(v))\phi _{j}(v)\nu _{T}(dv)+1-\frac{c^{j}}{c_{\mu }^{j}}+\log
\left( \frac{c^{j}}{c_{\mu }^{j}}\right) .
\end{equation*}%
Combining the last display with (\ref{Eqn: re}), we have $\int_{\mathbb{X}%
_{T}}l(g(v))\phi _{j}(v)\nu _{T}(dv)\leq N$, for all $j$. Sending $%
j\rightarrow \infty $, we see that $g\in S^{N}$. The result follows.
\end{proof}

\subsection{Proof of Theorem \protect\ref{Thm:LDP2}}

\begin{proof}
Proof follows by modifying arguments for the lower bound and upper bound in
the proof of Theorem 4.2 of \cite{BDM09}.

\emph{Lower Bound}. Following the proof of Theorem 2.8 in \cite{BDM09}, it
is easy to see that $-\epsilon \log \mathbb{\bar{E}}\left( e^{-\epsilon
^{-1}F(Z^{\epsilon })}\right) $ is bounded below (actually equal to)
\begin{equation}
\inf_{\varphi \in \tilde{\mathcal{U}}}\bar{\mathbb{E}}\left[ {L}_{T}(\varphi
)+F\circ \mathcal{G}^{\epsilon }\left( \epsilon N^{\epsilon ^{-1}\varphi
}\right) \right] ,  \label{eq:131}
\end{equation}%
where $\tilde{\mathcal{U}}=\cup _{N\geq 1}\tilde{\mathcal{U}}^{N}$. The rest
of the proof for the lower bound is as in Theorem 4.2 of \cite{BDM09}.

\emph{Upper Bound}. Fix $\delta \in (0,1)$ and $\phi _{0}\in \mathbb{U}$
such that%
\begin{equation*}
I(\phi _{0})+F(\phi _{0})\leq \inf_{\phi \in \mathbb{U}}(I(\phi )+F(\phi
))+\delta \mbox{.}
\end{equation*}%
Choose $g\in \mathbb{S}_{\phi _{0}}$ such that ${L}_{T}(g)\leq I(\phi
_{0})+\delta $. Note that $g\in \mathbb{S}_{\phi _{0}}$ implies $\phi _{0}=%
\mathcal{G}^{0}\left( \nu _{T}^{g}\right) $. Define
\begin{numcases}{g_n(t,x)=}
\left[g(t,x)\vee \frac{1}{n}\right]\wedge n & for $x \in K_n$,\nonumber\\
1 & else.\nonumber
\end{numcases}Then $g_{n}\in \bar{\mathcal{A}}_{b,n}\subset \bar{\mathcal{A}}%
_{b}$. By the monotone convergence theorem, $L_{T}(g_{n})\uparrow L_{T}(g)$.

Recalling from the proof of the lower bound that $-\epsilon \log \mathbb{%
\bar{E}}\left( \exp \left( -\epsilon ^{-1}F(Z^{\epsilon })\right) \right) $
equals the expression in \eqref{eq:131},
\begin{eqnarray*}
\limsup_{\epsilon \rightarrow 0}-\epsilon \log \mathbb{\bar{E}}\left(
e^{-\epsilon ^{-1}F(Z^{\epsilon })}\right)  &\leq &{L}_{T}(g_{n})+\limsup_{%
\epsilon \rightarrow 0}\bar{\mathbb{E}}\left[ F\circ \mathcal{G}^{\epsilon
}\left( \epsilon N^{\epsilon ^{-1}g_{n}}\right) \right]  \\
&\leq &{L}_{T}(g_{n})+F\circ \mathcal{G}^{0}\left( \nu _{T}^{g_{n}}\right) ,
\end{eqnarray*}%
where the last inequality follows on observing that since $g_{n}\in \tilde{%
\mathcal{U}}^{N}$ for some $N$, we have by assumption that, for each fixed $n
$, $\mathcal{G}^{\epsilon }(\epsilon N^{\epsilon ^{-1}g_{n}})\Rightarrow
\mathcal{G}^{0}\left( \nu _{T}^{g_{n}}\right) $, as $\epsilon \rightarrow 0$%
. Sending $n\rightarrow \infty $, we have
\begin{eqnarray*}
\limsup_{\epsilon \rightarrow 0}-\epsilon \log \mathbb{\bar{E}}\left(
e^{-\epsilon ^{-1}F(Z^{\epsilon })}\right)  &\leq &{L}_{T}(g)+F\circ
\mathcal{G}^{0}\left( \nu _{T}^{g}\right)  \\
&\leq &I(\phi _{0})+\delta +F\circ \mathcal{G}^{0}\left( \nu _{T}^{g}\right)
\\
&=&I(\phi _{0})+F(\phi _{0})+\delta  \\
&\leq &\inf_{\phi \in \mathbb{U}}(I(\phi )+F(\phi ))+2\delta \mbox{.}
\end{eqnarray*}%
Since $\delta \in (0,1)$ is arbitrary the desired upper bound follows. This
completes the proof of the theorem.
\end{proof}

\subsection{Proof of Remark \protect\ref{rmk:exp}}

\begin{proof}
Let $E\in \mathcal{B}(\mathbb{X}_{T})$ be such that $\nu _{T}(E)<\infty $.
Fix $\delta _{2}\in (0,\infty) $, and define $F=\{(s,v)\in \mathbb{X}%
_{T}:||G(s,v)||_{0,-p}>\delta _{2}/\delta _{1}\}$. Then
\begin{equation*}
\begin{split}
\int_{E}e^{\delta _{2}||G(s,v)||_{0,-p}}\nu (dv)ds& =\int_{E\cap F}e^{\delta
_{2}||G(s,v)||_{0,-p}}\nu (dv)ds+\int_{E\cap F^{c}}e^{\delta
_{2}||G(s,v)||_{0,-p}}\nu (dv)ds \\
& \leq \int_{E\cap F}e^{\delta _{1}||G(s,v)||_{0,-p}^{2}}\nu
(dv)ds+e^{\delta _{2}^{2}/\delta _{1}}\int_{E\cap F^{c}}\nu (dv)ds \\
& \leq \int_{E}e^{\delta _{1}||G(s,v)||_{0,-p}^{2}}\nu (dv)ds+e^{\delta
_{2}^{2}/\delta _{1}}\nu _{T}(E)<\infty .
\end{split}%
\end{equation*}%
The remark follows.
\end{proof}

\subsection{Proof of Lemma \protect\ref{lem:exist}}

\begin{proof}
The proof proceeds through a standard Picard iteration argument. Define $%
x^{0}(t)=x_{0}$ for all $t\in \lbrack 0,T]$. Define $x^{n}(t)$ iteratively
as
\begin{equation*}
x^{n}(t)=x_{0}+\int_{0}^{t}a(s,x^{n-1}(s))ds+%
\int_{0}^{t}b(s,x^{n-1}(s))u(s,x^{n-1}(s))ds,\quad t\in \lbrack 0,T].
\end{equation*}%
Then
\begin{equation*}
\begin{split}
||x^{n}(t)||& \leq
||x_{0}||+\int_{0}^{t}||a(s,x^{n-1}(s))||ds+%
\int_{0}^{t}||b(s,x^{n-1}(s))u(s,x^{n-1}(s))||ds \\
& \leq ||x_{0}||+\int_{0}^{t}\kappa (1+||x^{n-1}(s)||)ds+\int_{0}^{t}\kappa
(1+||x^{n-1}(s)||)\sup_{y}||u(s,y)||ds \\
& \leq ||x_{0}||+\kappa (M+T)+\kappa
\int_{0}^{t}||x^{n-1}(s)||(1+\sup_{y}||u(s,y)||)ds.
\end{split}%
\end{equation*}%
Let $L=||x_{0}||+\kappa (M+T)$, $\alpha (s)=1+\sup_{y}||u(s,x)||$, and $%
\beta (t)=\int_{0}^{t}\alpha (s)ds$. Then a recursive argument shows that
for all $t\in \lbrack 0,T]$,
\begin{equation*}
||x^{n}(t)||\leq L+\kappa L\beta (t)+\frac{\kappa ^{2}L}{2}\beta
(t)^{2}+\cdots +\frac{\kappa ^{n}L}{n!}\beta (t)^{n},
\end{equation*}%
and thus
\begin{equation}
\sup_{n}\sup_{t\in \lbrack 0,T]}||x^{n}(t)||\leq Le^{\kappa \beta (T)}\leq
Le^{\kappa (M+T)}.  \label{eq:bnd}
\end{equation}%
Similarly
\begin{equation*}
\begin{split}
||x^{n}(t)-x^{n}(s)||& \leq
\int_{s}^{t}||a(s,x^{n-1}(r))||dr+%
\int_{s}^{t}||b(r,x^{n-1}(r))u(r,x^{n-1}(r))||dr \\
& \leq \kappa (1+Le^{\kappa (M+T)})(t-s)+\kappa (1+Le^{\kappa
(M+T)})\int_{s}^{t}\sup_{y}||u(r,y)||dr,
\end{split}%
\end{equation*}%
and therefore
\begin{equation*}
\lim_{\delta \rightarrow 0}\sup_{n}\sup_{|t-s|\leq \delta
}||x^{n}(t)-x^{n}(s)||=0.
\end{equation*}%
%
%
%
%
%
%
%
%
%
%
%
%
%
%
%
%
%
%
%
%
%
%
Together with \eqref{eq:bnd} shows that the sequence $\{x^{n}\}$ is
pre-compact in $C([0,T]:\mathbb{R}^{d})$. Let $x$ be a limit point of some
subsequence of $\{x^{n}\}$. Then using the continuity properties of the
functions $a$, $b$ and $u$ with respect to $x$ and the dominated convergence
theorem, it is easy to check that $x$ satisfies \eqref{eq:int}. The lemma
follows.
\end{proof}

\subsection{Proof of (\protect\ref{u_cont})}

\begin{proof}
Let $y_{n}\rightarrow y$, $y_{n},y\in \mathbb{R}^{d}$. We will like to show
that $u(s,y_{n})\rightarrow u(s,y)$ for a.e. $s\in \lbrack 0,T]$. Note that,
since $\psi \in S^{M}$, $\int_{[0,T]\times \mathbb{X}}l(\psi (s,v))\nu
(dv)ds\leq M.$ Thus there exists $\mathbb{T}_{1}\subset \lbrack 0,T]$, with $%
\lambda _{T}(\mathbb{T}_{1}^{c})=0$ and such that
\begin{equation*}
\int_{\mathbb{X}}l(\psi (s,v))\nu (dv)<\infty ,\quad \forall s\in \mathbb{T}%
_{1}.
\end{equation*}%
Also, from arguments similar to those in the proof of Lemma \ref{lem:Gg},
\begin{equation*}
\int_{\mathbb{X}_{T}}||g^{d}(s,v)||_{0}|\psi (s,v)-1|\nu (dv)ds<\infty .
\end{equation*}%
Consequently, there exists $\mathbb{T}_{2}\subset \lbrack 0,T]$, with $%
\lambda _{T}(\mathbb{T}_{2}^{c})=0$ and such that
\begin{equation}
\int_{\mathbb{X}}||g^{d}(s,v)||_{0}|\psi (s,v)-1|\nu (dv)<\infty ,\quad
\forall s\in \mathbb{T}_{2}.  \label{ab2032}
\end{equation}%
Let $\mathbb{T}=\mathbb{T}_{1}\cap \mathbb{T}_{2}$ and fix $s\in \mathbb{T}$%
. Define $F_{\beta }(s)=\{v\in \mathbb{X}:|\psi (s,v)-1|\leq \beta \}$ for $%
\beta \in (0,\infty )$. Then
\begin{equation*}
\begin{split}
u(s,y_{n})& =\int_{\mathbb{X}\cap F_{\beta }}\frac{g^{d}(s,y_{n},v)}{%
1+||y_{n}||}(\psi (s,v)-1)\nu (dv)+\int_{\mathbb{X}\cap F_{\beta }^{c}}\frac{%
g^{d}(s,y_{n},v)}{1+||y_{n}||}(\psi (s,v)-1)\nu (dv) \\
& =u_{1}(s,y_{n})+u_{2}(s,y_{n}).
\end{split}%
\end{equation*}%
From part (c) of Remark \ref{rmk:l}, for all $v\in F_{\beta }(s)$,
\begin{equation*}
|\psi (s,v)-1|^{2}\leq c_{2}(\beta )l(\psi (s,v)).
\end{equation*}%
Thus $[\psi (s,\cdot )-1]1_{F_{\beta }(s)}(\cdot )\in L^{2}(\mathbb{X},\nu ;%
\mathbb{R})$. From assumption (a) in Lemma \ref{lem:ddim} we now see that,
for all such $s$, $u_{1}(s,y_{n})\rightarrow u_{1}(s,y)$, as $n\rightarrow
\infty $.

For $u_{2}(s,y_{n})$, we have
\begin{equation*}
\left\vert \frac{g^{d}(s,y_{n},v)}{1+||y_{n}||}(\psi (s,v)-1)\right\vert
\leq ||g^{d}(s,v)||_{0}|\psi (s,v)-1|.
\end{equation*}%
From \eqref{ab2032}, the term on the right hand side is $\nu $-integrable.
Furthermore, $\nu (F_{\beta }^{c})\rightarrow 0$ from the super linear
growth of $l$. Thus $u_{2}(s,y_{n})$ converges to 0, uniformly in $n$, as $%
\beta $ goes to $\infty $. The term $u_{2}(s,y)$ can be treated in a similar
manner. Thus we have shown that, for all $s\in \mathbb{T}$, $%
u(s,y_{n})\rightarrow u(s,y)$. Since $\lambda _{T}(\mathbb{T}^{c})=0$, the
result follows.
\end{proof}

\subsection{Proof of (\protect\ref{eq:18a}) when $h$ is a bounded and
measurable function}

\begin{proof}
We can assume without loss of generality that $\int_{K}g\nu _{T}(dsdv)\neq 0$
and $\int_{K}g_{n}\nu _{T}(dsdv)\neq 0$, for all $n\geq 1$. Define
probability measures $\tilde{\nu}^{n}$ and $\tilde{\nu}$ as follows:
\begin{equation*}
\tilde{\nu}^{n}(\cdot) =\frac{\nu _{T}^{g^{n}}(\cdot \cap K)}{m_{n}},\quad
\tilde{\nu}(\cdot)=\frac{\nu^g _{T}(\cdot \cap K)}{m}
\end{equation*}%
where $m_{n}=\int_{K}g_{n}\nu _{T}(dsdv)$ and $m=\int_{K}g\nu _{T}(dsdv)$.
If $\theta(\cdot) =\frac{\nu _{T}(\cdot \cap K)}{\nu _{T}(K)}$, then $\theta
$ is also a probability measure. We have
\begin{equation*}
\begin{split}
R(\tilde{\nu}^{n}||\theta )& =\int_{K}\log \left( \frac{\nu_T (K)}{m_{n}}%
g_{n}\right) \frac{1}{m_{n}}g_{n}\nu _{T}(dsdv) \\
& =\frac{1}{m_{n}}\int_{K}(l(g_{n})+g_{n}-1)\nu _{T}(dsdv)+\log \frac{\nu_T
(K)}{m_{n}} \\
& \leq \frac{N}{m_{n}}+1-\frac{\nu_T (K)}{m_{n}}+\log \frac{\nu_T (K)}{m_{n}}%
.
\end{split}%
\end{equation*}%
Noting that $m_{n}\rightarrow m$, we have that there exists constant $\alpha
$ such that $\sup_{n\in \mathbb{N}}R(\tilde{\nu}^{n}||\theta )\leq \alpha
<\infty $. Also note that $\tilde{\nu}^{n}$ converges weakly to $\tilde{\nu}$%
. From Lemma 2.8 of \cite{bodu98}, we have
\begin{equation*}
\frac{1}{m_{n}}\int_{[0,T]\times K}h(s,v)g_{n}(s,v)\nu _{T}(dvds)\rightarrow
\frac{1}{m}\int_{[0,T]\times K}h(s,v)g(s,v)\nu _{T}(dvds),
\end{equation*}%
which proves \eqref{eq:18a}.
\end{proof}

\subsection{Proof of It\^{o}'s formula in (\protect\ref{eq:ito})}

\begin{proof}
Here we will give the proof for a simpler case when $X_{t}$ satisfies the
following integral equation, the proof of \eqref{eq:ito} being very similar
to this case:
\begin{equation*}
{X}_{t}=X_{0}+\int_{0}^{t}A(s,{X}_{s})ds+\int_{0}^{t}\int_{\mathbb{X}}G(s,{X}%
_{s-},v)\tilde{N}(dsdv).
\end{equation*}%
For $j\in \mathbb{N}$,
\begin{equation*}
{X}_{t}[\theta _{p}\phi _{j}]=X_{0}[\theta _{p}\phi _{j}]+\int_{0}^{t}A(s,{X}%
_{s})[\theta _{p}\phi _{j}]ds+\int_{0}^{t}\int_{\mathbb{X}}G(s,{X}%
_{s-},v)[\theta _{p}\phi _{j}]\tilde{N}(dsdv).
\end{equation*}%
Note that
\begin{equation*}
{X}_{t}[\theta _{p}\phi _{j}]=\langle {X}_{t},\phi _{j}\rangle _{-p}=||\phi
_{j}||_{-p}\langle {X}_{t},\phi _{j}^{-p}\rangle _{-p},
\end{equation*}%
so
\begin{equation*}
\sum_{j=1}^{\infty }||\phi _{j}||_{p}^{2}({X}_{t}[\theta _{p}\phi
_{j}])^{2}=\sum_{j=1}^{\infty }\langle {X}_{t},\phi _{j}^{-p}\rangle
_{-p}^{2}=||X_{t}||_{-p}^{2}.
\end{equation*}

If $\xi _{j}(t)={X}_{t}[\theta _{p}\phi _{j}]$, then $\xi _{j}(t)$ satisfies
\begin{equation*}
\xi _{j}(t)=\xi _{j}(0)+\int_{0}^{t}a^{j}(s)ds+\int_{0}^{t}\int_{\mathbb{X}%
}b^{j}(s,v)\tilde{N}(dsdv).
\end{equation*}%
where $a^{j}(s)=A(s,{X}_{s})[\theta _{p}\phi _{j}]$ and $b^{j}(s,v)=G(s,{X}%
_{s-},v)[\theta _{p}\phi _{j}]$. Applying It\^{o}'s formula (cf. Theorem
2.5.1 of \cite{Ikeda}) to the real valued semimartingale $\xi _{j}(t)$, we
have
\begin{equation}
\begin{split}
\xi _{j}^{2}(t)&= \xi _{j}^{2}(0)+2\int_{0}^{t}a^{j}(s)\xi
_{j}(s)ds+2\int_{0}^{t}\int_{\mathbb{X}}b^{j}(s,v)\xi _{j}(s-)\tilde{N}(dsdv)
\\
&\quad +\int_{0}^{t}\int_{\mathbb{X}}[b^{j}(s,v)]^{2}\tilde{N}%
(dsdv)+\int_{0}^{t}\int_{\mathbb{X}}[b^{j}(s,v)]^{2}\nu (dv)ds.
\end{split}
\label{eq:ito2}
\end{equation}

Note that $||X_t||_{-p}^2=\sum_{j=1}^\infty ||\phi_j||_{p}^2\xi_j^2(t)$. So
for the second term in \eqref{eq:ito2}, we have
\begin{equation*}
\begin{split}
\sum_{j=1}^\infty ||\phi_j||_{p}^2 a^j(s)\xi_j(s)&= \sum_{j=1}^\infty
||\phi_j||_{p}^2 A(s,{X}_s)[\theta_p\phi_j]{X}_s[\theta_p\phi_j] \\
&= A(s,{X}_s)\left[\sum_{j=1}^\infty ||\phi_j||_{p}^2{X}_s[\theta_p\phi_j]%
\theta_p\phi_j\right] \\
&= A(s,{X}_s)\left[\sum_{j=1}^\infty ||\phi_j||_{p}^2\langle {X}_s,\phi_j
\rangle_{-p}||\phi_j||_{-p}^2\phi_j\right] \\
&= A(s,{X}_s)\left[\sum_{j=1}^\infty \langle {X}_s,\phi_j^{-p}
\rangle_{-p}\phi_j^p\right] \\
&= A(s,{X}_s)[\theta_pX_s].
\end{split}%
\end{equation*}
Also, we have
\begin{equation*}
\begin{split}
\sum_{j=1}^\infty ||\phi_j||_{p}^2 b^j(s,v)\xi_j(s-)&= \sum_{j=1}^\infty
||\phi_j||_{p}^2 G(s, {X}_{s-},v)[\theta_p\phi_j]{X}_{s-}[\theta_p\phi_j] \\
&= \sum_{j=1}^\infty ||\phi_j||_{p}^2 \langle G(s, {X}_{s-},v),
\phi_j\rangle_{-p}\langle {X}_{s-}, \phi_j\rangle_{-p} \\
&= \sum_{j=1}^\infty \langle G(s, {X}_{s-},v),
\phi_j^{-p}\rangle_{-p}\langle {X}_{s-}, \phi_j^{-p}\rangle_{-p} \\
&= \langle G(s, {X}_{s-},v), {X}_{s-}\rangle_{-p}.
\end{split}%
\end{equation*}
Finally, notice that
\begin{equation*}
\begin{split}
\sum_{j=1}^\infty ||\phi_j||_{p}^2 [b^j(s,v)]^2&= \sum_{j=1}^\infty
||\phi_j||_{p}^2 \left(G(s, {X}_{s-},v)[\theta_p\phi_j]\right)^2 \\
&= \sum_{j=1}^\infty ||\phi_j||_{p}^2 \left(\langle G(s, {X}_{s-},v),
\phi_j\rangle_{-p}\right)^2 \\
&= \sum_{j=1}^\infty \left(\langle G(s, {X}_{s-},v), \phi_j^{-p}\rangle_{-p}
\right)^2 \\
&= ||G(s, {X}_{s-},v)||^2_{-p}.
\end{split}%
\end{equation*}
Combining the above equalities with \eqref{eq:ito2}, we have
\begin{equation*}
\begin{split}
||X_t||_{-p}^2 &= ||X_0||_{-p}^2 + 2\int_0^t A(s,{X}_s)[\theta_pX_s] ds +
2\int_0^t \int_\mathbb{X} \langle G(s, {X}_{s-},v), {X}_{s-}\rangle_{-p}%
\tilde{N}(dsdv) \\
&\quad +\int_0^t \int_\mathbb{X} ||G(s, {X}_{s-},v)||^2_{-p}\tilde{N}(dsdv)
+\int_0^t \int_\mathbb{X} ||G(s, {X}_{s-},v)||^2_{-p}\nu(dv)ds.
\end{split}%
\end{equation*}
\end{proof}


\vspace{\baselineskip} \bigskip

\noindent \textsc{A. Budhiraja and J. Chen\newline
Department of Statistics and Operations Research\newline
University of North Carolina\newline
Chapel Hill, NC 27599, USA }

\vspace{\baselineskip}

\noindent \textsc{P. Dupuis\newline
Lefschetz Center for Dynamical Systems\newline
Division of Applied Mathematics\newline
Brown University\newline
Providence, RI 02912, USA }

\end{document}